\documentclass[12pt,reqno]{amsart}

\usepackage{amsmath,amssymb,latexsym,soul,cite,mathrsfs}
\usepackage{mathtools}
\usepackage{enumerate}
\usepackage{lipsum}

\usepackage{color,graphicx}
\usepackage[colorlinks=true,urlcolor=blue,
citecolor=red,linkcolor=blue,linktocpage,pdfpagelabels,
bookmarksnumbered,bookmarksopen]{hyperref}
\usepackage[english]{babel}

\usepackage[left=2.5cm,right=2.5cm,top=2.5cm,bottom=2.5cm]{geometry}

\usepackage[colorinlistoftodos]{todonotes}
\usepackage{dsfont}
\makeatletter
\providecommand\@dotsep{5}
\def\listtodoname{List of Todos}
\def\listoftodos{\@starttoc{tdo}\listtodoname}
\makeatother

\numberwithin{equation}{section}

\newtheorem{theorem}{Theorem}[section]
\newtheorem{proposition}[theorem]{Proposition}
\newtheorem{lemma}[theorem]{Lemma}
\newtheorem{corollary}[theorem]{Corollary}

\newtheorem{remark}{Remark}
\newtheorem{definition}[theorem]{Definition}

\begin{document}
	
	\title[Multiple critical points theorems]{Multiple critical points theorems for a class of nonsmooth functionals and applications to problems driven by 1-Laplacian and discontinuous nonlinearities}

	\author{Ismael Sandro da Silva}
	\author {Marcos T. Oliveira Pimenta}
	\author{Pedro Fellype Pontes}

	\address[Ismael Sandro da Silva]
	{\newline\indent Universidade Estadual da Paraíba
		\newline\indent
		Centro de Ciências Exatas e Sociais Aplicadas
		\newline\indent
		 Patos 58700-070 -- Brazil}
	\email{\href{ismaels@servidor.uepb.edu.br}{ismaels@servidor.uepb.edu.br}}
	
	\address[Marcos Tadeu Oliveira Pimenta]
	{\newline\indent Universidade Estadual Paulista
		\newline\indent
		Departamento de Matemática e Computação
		\newline\indent
		Presidente Prudente 19060-900 -- Brazil}
	\email{\href{marcos.pimenta@unesp.br}{marcos.pimenta@unesp.br}}
	
	\address[Pedro Fellype Pontes]
	{\newline\indent Jilin University
		\newline\indent
		School of Mathematics
		\newline\indent
		Changchun 130012 -- People's Republic of China}
	\email{\href{fellype.pontes@gmail.com}{fellype.pontes@gmail.com}}

	\pretolerance10000
	
	
	\begin{abstract}
	\noindent  In this paper, we present a novel approach to investigate the existence of multiple critical points for a class of nonsmooth functionals. This method provides a robust framework to analyze the existence of solutions for problems involving the 
	$1$-Laplacian operator with discontinuous nonlinearities. Our results contribute to advancing the study of nonsmooth variational problems, by establishing new nonsmooth multiple critical point theorems.  
\end{abstract}

\thanks{Pedro Fellype Pontes was partially supported by NSFC (W2433017), and BSH (2024-002378)}

\maketitle

\section{Introduction}\label{Sec1}

In recent decades, nonlinear partial differential equations with discontinuous nonlinearities have garnered significant attention from researchers. This interest is largely due to their connection to numerous free boundary problems arising in mathematical physics. Examples include the obstacle problem \cite{C2, Y}, the Goldshtik problem for separated flows of incompressible fluids \cite{G, P2}, superconductivity phenomena \cite{P}, the Elenbaas problem related to electric discharge origination \cite{AT, E}, among others. To address these problems, various methods have been employed, such as variational techniques for nondifferentiable functionals, the method of lower and upper solutions, global branching theory, and others. In addition to the previously mentioned studies, we also refer to \cite{CBG,AS,AB,BT,SPS}, and references therein.

Additionally, problems involving the $1$-Laplacian operator have attracted considerable attention in recent years. This highly singular and degenerate operator plays a central role in several applied contexts, including image processing, material science, and minimal surfaces. From the mathematical perspective, a systematic analysis of the $1$-Laplacian can be found in the works \cite{AP,FP,SL,HHL,PV} and the references therein. Although these papers are primarily devoted to the analytical foundations of the operator, their introductions provide detailed discussions of the main applications motivating the theory. 

Two primary approaches are commonly adopted to study such problems. The first relies on analyzing the associated energy functional, which is typically defined in the space $BV(\Omega)$ -- a non-reflexive space (see \eqref{BVdef} below for a definition). In this framework, the energy functional can often be expressed as the difference of a convex and locally Lipschitz functional and a $C^1$ one. The second approach approximates the $1$-Laplacian operator using the $p$-Laplacian operator for $p>1$. By solving the corresponding $p$-Laplacian problem and studying the behavior of solutions as $p \to 1^+$, researchers can gain insights into the original $1$-Laplacian problem. 

Considering this, it is natural to explore problems governed by the $1$-Laplacian operator where the nonlinearity permits discontinuities. Two recent works, \cite{PSJ} and \cite{PJS}, have addressed this topic. The first study examined the problem $ -\Delta_1 u = H(u-\beta)|u|^{q-2}u $, where $ H(\cdot) $ denotes the Heaviside function, $\beta > 0$ is a parameter, and $ 1 < q < 1^* $. The authors proved that for each $\beta > 0$, there exists at least one non-negative solution $u_\beta$.

The second work extended this analysis to the problem $-\Delta_1 u = \lambda H(u-\beta)|u|^{q-2}u + |u|^{1^*-2}u,$
where $\lambda > 0$ is a parameter. It was shown that there exist thresholds $\lambda_0, \beta_0 > 0$ such that the problem admits a non-negative solution $u_{\beta, \lambda}$ for all $\lambda > \lambda_0$ and $0 < \beta < \beta_0$. Both studies utilized the $p$-Laplacian approximation to manage the $1$-Laplacian operator and used the ideas of generalized gradient, due to Chang \cite{C}, to handle the discontinuity in the nonlinearity.

Notably, these works also addressed the stability of solutions, analyzing the behavior as the discontinuity parameter $\beta$ approaches zero. They demonstrated the existence of a limiting solution for the sequence $(u_\beta)$, which corresponds to the continuous case where $H(t) \equiv 1$. This result bridges the discontinuous and continuous settings, offering insights into the transition between the two regimes. However, in neither case is the multiplicity of solutions discussed. To the best of our knowledge, these are the only works in the literature that involve problems governed by the 1-Laplacian with discontinuities in nonlinearity.

Building upon the discussions presented earlier, we investigate the existence and multiplicity of solutions for the following boundary value problem:
	\begin{equation}\label{1.1}
		\left\{ \begin{array}{rclcc} \displaystyle
			-\Delta_1 u & =& f_{\lambda,a}(u), &\mbox{in}& \Omega; \\
			u& =& 0, & \mbox{on}& \partial\Omega,
		\end{array}
		\right.
	\end{equation}
where $\Omega \in \mathbb{R}^N$, with $N\ge2$, is a bounded domain and $\lambda,a>0$ are parameters. The study focuses on two specific forms of the nonlinearity $f_{\lambda,a}$:
	\begin{itemize}
		\item[\underline{\it Case 1:}] $f_{\lambda,a}(t) = \lambda\mbox{sign}(t) + H(t-a) |t|^{q-2}t$;
		\vspace{.3cm}
		
		\item[\underline{\it Case 2:}] $f_{\lambda,a}(t) = \lambda H(t-a) |t|^{q-2}t + |t|^{1^*-2}t$.
	\end{itemize}

In both cases a new technique is used, more precisely we have developed abstract results related to the class of functionals $I=\Phi+\Psi:X\to(-\infty,\infty]$ that are a sum of a locally Lipschitz functional with a lower semicontinuous and convex functional on the Banach space $X$. Such class of functional have been originally introduced in the literature by Motreanu and Panagiotopoulos in \cite{Motreanu} and generalizes simultaneously the single case of locally Lipschitz functional due to Chang \cite{C2} and of Szulkin type-functionals, which are of the form $\Phi_2+\Psi_2:X\to(-\infty,\infty]$, where $\Phi_2$ is a $C^1$-functional and $\Psi_2$ is as $\Psi$ above. See \cite{Szulkin} for a complete study related to Szulkin type-functionals.

Recently, Alves, Bisci and da Silva \cite{ABS} have established some theoretical results that guarantee the existence of multiple critical point for Szulkin type-functionals. In this text, we have extended some multiple critical points results proved in \cite{ABS} for the aforementioned class of functionals $I=\Phi+\Psi$ explored in \cite{Motreanu}. Besides of the abstract theorems related with the multiplicity of critical points for nondifferentiable functionals, our work contains important contributions in the sense of the nonsmooth analysis (see, e.g, Lemmas \ref{caracPS} and \ref{TL} in the sequel).

The theoretical results obtained in our study will be employed in order to get the multiplicity of solution for the problem (\ref{1.1}). Usually, problems involving the 1-Laplacian operator and discontinuous nonlinearities have been addressed by the aforementioned limit process as $p\to 1^+$ dealing with the $p$-Laplacian operator. We refer to the reader the pioneering works \cite{PJS, PSJ}. In general, the approach based on a limit procedure involves several technical difficulties in order to get the multiplicity of existence for (\ref{1.1}). In fact, the estimates involving the $p$-Laplacian operator should be carefully analyzed with respect the dependence on $p>1$. In our approach, we have overcame these difficulties because the theory for nonsmooth functionals allows us to work directly on the space of the functions of bounded of variation $BV(\Omega)$. 

\medskip
Let us recall that, given $I= \Phi+\Psi:X\to(-\infty,\infty]$, with $\Phi$ a locally Lipschitz function and $\Psi$ a convex and lower semicontinuous functional ($\Psi \not\equiv \infty$), a critical point for $I$ means a point $u\in X$ such that $I(u)<\infty$ and
	$$\Phi^\circ(u,v-u)+\Psi(v)-\Psi(u)\geq 0,\quad\forall v\in X,$$
where $\Phi^\circ(u,\cdot)$ designates the generalized gradient of $\Phi$ at $u$ (for more details see Section \ref{Sec2}). This concept enables us to identify a critical point of the nonsmooth energy functional associated with (\ref{1.1}) as a solution of (\ref{1.1}) in the \textit{sense of subdifferentials} such as in \cite{PJS} (see Section \ref{Sec4} below for further details). Hence, one can gets multiple solutions for (\ref{1.1}) by finding for multiple critical points of the energy functional of the problem. 
	
Following upon this direction, by borrowing the ideas of the $G$-index theory introduced in \cite{ABS}, we have proved in this paper an equivariant version of the symmetric Mountain Pass theorem for nonsmooth functionals as $I$ above that verify a symmetry condition with respect the action of a topological group $G$ on $X$ (Theorem \ref{SP1} and Corollary \ref{SP2} in the sequel). We point out our results improve the statements in \cite[Theorem 4.4]{Szulkin}, \cite[Corollary 3.6]{Motreanu}, and \cite[Theorem 3.9]{ABS}. With respect to the technical results in the sense of the nonsmooth analysis, we mention that the Lemma \ref{caracPS} plays a valuable role in understanding of behavior of the so-called $(\rm PS)$ sequences associated with the energy functional of (\ref{1.1}).

\medskip
Related to the problem in (\ref{1.1}) our main applications are the following, in the Case 1:

\begin{theorem}\label{Main1}
	There exists $\lambda_0>0$ such that $I_{\lambda,a}$ has a sequence of critical points $\big(u_{\lambda,a}^{(n)}\big)_{n \in \mathbb{N}}$ for any $\lambda \in (0,\lambda_0)$. Consequently, the problem $(P_{\lambda,a})$ has infinitely many solutions for $\lambda \in (0,\lambda_0)$. Moreover, there exists $a_0>0$ small enough such that each solution $u_{\lambda,a}^{(n)}$  satisfies
	$$\left|\,[|u_{\lambda,a}^{(n)}| > a]\,\right|>0, \eqno{(M_1)}$$
	for any $a \in (0,a_0)$.
\end{theorem}

In Case 2:

\begin{theorem}\label{Main2}
	For any $n \in \mathbb{N}$, there exists $\lambda_n > 0$ such that for every $\lambda \in [\lambda_n, \infty)$, there exists $a_\lambda > 0$ satisfying the following: the problem $(Q_{\lambda,a})$ admits at least $n$ nontrivial solutions for any $a \in (0, a_\lambda)$. Moreover, for each $\lambda \in [\lambda_n, \infty)$, there exists $0 < a_0 \leq a_\lambda$ such that:  
	$$|[|u| > a]| > 0, \quad \forall \, a \in (0, a_0), \eqno{(M_2)}$$
	where $u$ is a solution of $(Q_{\lambda,a})$.
\end{theorem}

We would like to emphasize that in the both Case 1 and 2 we have complemented the situations studied in \cite{PJS, PSJ}. We also point out that in the Case 1 we are dealing with a multiplicity of discontinuities, namely, at 0 and at $a$. This fact marks a novelty for this class of problems. To the best our knowledge, this is the first time the multiplicity of solution is analyzed for quasilinear involving concomitantly the $1$-Laplacian operator and discontinuous nonlinearity

The article is essentially organized as follows: in Section \ref{Sec2}, we present the main tools and technical results related to non-smooth analysis. Section \ref{Sec3} is dedicated to proving our main theoretical results. Finally, in Section \ref{Sec4}, we analyze the existence of multiple solutions for \eqref{1.1} in Cases 1 and 2; specifically, Subsection \ref{Ap1} deals with Case 1, while Subsection \ref{Ap2} is dedicated to Case 2.

\section{On nonsmooth analysis}\label{Sec2}
Before presenting our main results, this section is devoted to recall some notions related to the study of nonsmooth analysis and the generalized Critical Point Theory. For further details and proofs we refer the works \cite{C,CLM,Clarke, Clarke1, Motreanu}. Hereafter, we will consider $(X,\|\cdot \|) $ and $(X^*,\|\cdot \|_*)$ a Banach space and its topological dual space. As usual, we will designate by $\langle \cdot, \cdot \rangle$ the duality pairing between $X$ and $X^*$.

To begin with, we recall that a real-valued functional $\Phi:X \rightarrow \mathbb{R} $ is called \textit{locally Lipschitz continuous} (briefly $\Phi\in {\textrm{Lip}}_{\textrm{loc}}(X, \mathbb{R})$) when to every $u \in X$ there
correspond a neighbourhood $V\coloneqq V_u$ of $u$ and a constant $L\coloneqq L(u) >0$ such that
\begin{equation*}
|\Phi(v)-\Phi(w)| \leq L \|v-w\|,\,\,\;\; \forall v,w \in V.
\end{equation*}
The \textit{generalized directional derivative} of $\Phi\in {\textrm{Lip}}_{\textrm{loc}}(X, \mathbb{R})$ at $u$ along the direction $v \in X$ is defined by
\begin{equation*}
\Phi^{\circ}(u;v)\coloneqq \limsup_{w\rightarrow u,\, t \rightarrow 0^+}\frac{\Phi(w+tv)-\Phi(w)}{t}.
\end{equation*}
The \textit{generalized gradient} of the function $\Phi\in {\textrm{Lip}}_{\textrm{loc}}(X, \mathbb{R})$ in $u$ is the set
\begin{equation*}
\partial \Phi(u)= \{\phi \in X^{*}\;\; : \; \; \Phi^{\circ}(u;v) \geq \left<\phi, v\right>,\,\forall\; v \in X \}.
\end{equation*}
It is well known that $\partial \Phi(u)$ turns out nonempty, convex and weak*-compact  subset of $X^*$, for any $u \in X$. We say that $u \in X$ is critical point of $\Phi$ if $0 \in \partial \Phi(u)$. We list below some important properties of the generalized gradient of $\Phi$.
\begin{proposition}\label{21}
	Given $\Phi_1$, $\Phi_2 \in {\rm Lip_{loc}}(X,\mathbb{R})$, the following items hold.
	\begin{enumerate}
		\item[$i)$] For any $u\in X$, $\Phi_1^{\circ}(u,\cdot)$ is a continuous and convex function with $\Phi_1^{\circ}(u,0)=0$.
		\item[${ii)}$] the map $(u,v) \mapsto \varphi^\circ(u,v)$ is an upper semicontinuous functional, i.e, if \linebreak $(u_j,v_j) \rightarrow (u,v)$, then
		$$ \limsup \Phi_1^\circ(u_j,v_j) \leq \Phi_1^\circ(u,v);$$
		\item[$iii)$] If $\Phi_1$ is continuously differentiable, then $\Phi_1^{\circ}(u,v)=\langle \Phi_1'(u),v \rangle$ and $\partial \Phi_1(u)=\{\Phi_1'(u)\}$.
		\item[$iv)$] For any $u \in X$, $\lambda \in \mathbb{R}$, it holds
		$\partial(\Phi_1+\Phi_2)(u)\subset \partial \Phi_1(u)+\partial \Phi_2(u),$
		and
		$\partial \Phi_1(\lambda u)=\lambda \Phi_1(u)$.
	\end{enumerate}
\end{proposition} 

Given a convex functional $\psi:X\rightarrow\mathbb{R}$, the \textit{subdifferential} of $\psi$ at $u$ is the set
\label{sd}\begin{equation*}
\partial \psi(u)\coloneqq\{\phi \in X^* \;\; : \;\; \psi(v)-\psi(u)\geq \left<\phi,v-u\right>, \forall\,v \in X\}.
\end{equation*}
Hence, for a functional $\Phi \in {\rm Lip_{loc}}(X,\mathbb{R})$, using the item-$i)$ of the Proposition \ref{21}, the following characterization holds.
\begin{equation}\label{sd1}
	\partial \Phi(u) = \partial \Phi^{\circ}(u,\cdot)(0),
\end{equation}
where the right side of the equality denotes the subdifferential of the convex function $\Phi^{\circ}(u,\cdot)$ at 0. Still related with locally Lipschtz functionals, we present below an useful result involving generalized gradients.
\begin{lemma}\label{incgrad1}
	Consider $\mathcal{F} : X\to\mathbb{R}$ and $F:Y \to \mathbb{R}$, where $Y$ is a dense subspace of $X$ and
		$$ F \equiv \mathcal{F}\mid_{Y}.$$
	Then, the inclusion $\partial F(u) \subseteq \partial \mathcal{F}(u)$ holds for any $u \in Y$.
\end{lemma}

Now, we introduce the main class of functionals to be considered on this paper. In the sequel, we will say that a functional $I:X\longrightarrow (-\infty,\infty]$ is a \textit{Motreanu-Panagiotopoulos functional} if $I$ is of the form
\\ \\
$(H_1):$ $I=\Phi+\Psi: X\longrightarrow (-\infty, \infty]$, with $\Phi \in {\rm Lip_{loc}} (X,\mathbb{R})$ and $\Psi$ is a convex lower semicontinuous (l.s.c) functional and proper, i.e., $\Psi \not\equiv \infty$.
\\

We will also write briefly $I \in (H_1)$ to mean that $I$ is a Motreanu-Panagiatopoulos functional. The \textit{effective domain} of a Motreanu-Panagiatopoulos functional is the set
$$D(I)\coloneqq\{u\in X;\, I(u)<\infty\}.$$
In this way, it is easily seen that $D(I)=D(\Psi)$.

Next, we present the main tools related with the Critical Point Theory for functionals $I\in (H_1)$. 

\begin{definition}\label{cp}
	Suppose that $I\in (H_1)$. Then,	
	\begin{enumerate}
		\item[$i)$] A point $u \in X$ is said to be a critical point of $I$ if $u \in D(I)$ and
		\begin{equation*}
		\Phi^{\circ}(u,v-u) + \Psi(v)-\Psi(u) \geq 0,\,\,\, \,\,\forall v \in X,
		\end{equation*}
		or, equivalently (see \cite[Proposition 2.183]{CLM}), if
		$$0 \in \partial \Phi(u) +\partial \Psi(u).$$
		\item[$ii)$] A sequence $(u_n)$ is called a Palais-Smale sequence for $I$ at level $c \in \mathbb{R}$ $($briefly $(\rm PS)_c$ or $(\rm PS)$ sequence$)$ if $I(u_n) \rightarrow c$ and
		\begin{equation*}
		\Phi^{\circ}(u_n,v-u_n) + \Psi(v)-\Psi(u_n) \geq -\varepsilon_n\|v-u_n\|,\,\,\,\,\, \forall v \in X,
		\end{equation*}
		with $\varepsilon_n\rightarrow 0^+$, which is equivalent to (see \cite[Proposition 3.1]{Motreanu}) $I(u_n)\rightarrow c$ and 
		$$\Phi^{\circ}(u_n,v-u_n) + \Psi(v)-\Psi(u_n) \geq \langle w_n,v-u_n\rangle,\,\,\,\,\, \forall v \in X,$$
		with $w_n\rightarrow 0$ in $X^*$. 
		\item[$iii)$] We will say that $I$ satisfies the Palais-Smale condition $($briefly $(\rm PS)_c$ condition$)$ at level $c\in \mathbb{R}$ when each $(\rm PS)$ sequence $(u_n)$ at a level $c$ has a convergent subsequence. If $I$ verifies the $(\rm PS)$ condition for all level $c$, we say simply that $I$ satisfies the $(\rm PS)$ condition.
		\end{enumerate}
	\end{definition}
	
	\begin{remark}
		{\rm By taking $\Psi\equiv 0$ in the previous definition, the notions of critical point and $(\rm PS)$ sequences for a functional $I\in (H_1)$ generalize the singular case of functionals $\Phi \in {\rm Lip_{loc}}(X,\mathbb{R})$ due to Chang \cite{C}, as well as the classical case of $C^1$-functionals.} 
	\end{remark}

	Next, we prove an important characterization of $(\rm PS)$ sequences to functionals $I\in (H_1)$.
	
	\begin{lemma}\label{caracPS}
		Suppose that $I\in (H_1)$. Then, a sequence $(u_n)$ is a $(\rm PS)_c$ sequence for $I$ if, and only if, $I(u_n)\rightarrow c$ and there exists a sequence $(w_n)$ in $X^*$, $w_n \rightarrow 0$ in $X^*$, satisfying
		$$ w_n \in \partial \Phi(u_n)+\partial \Psi(u_n),\,\,\,n\in \mathbb{N}.$$
	\end{lemma}
	\begin{proof}
		Let $(u_n)$ be an arbitrary $(\rm PS)_c$ sequence of $I$. Thus, for some sequence $(w_n)$ in $X^*$, it holds
		$$ \Phi^{\circ}(u_n,v-u_n) + \Psi(v)-\Psi(u_n) \geq \langle w_n,v-u_n\rangle,\,\,\, \forall v \in X,$$
		with $w_n\rightarrow 0$ in $X^*$. Therefore, setting, for each $n\in \mathbb{N}$ and $w\in X$, $v=w+u_n$ and $G_n(\cdot)\coloneqq-\langle w_n,\cdot \rangle$, we have
		$$F_n(w)\coloneqq\Phi^{\circ}(u_n,w) + \Psi(w+u_n)-\Psi(u_n)+G_n(w) \geq 0,\,\,\,\forall w\in X.$$ 
		It can quickly seen that the map $F_n$ satisfies the following properties.\\ \\
		$i)$ $F_n$ is a convex function for every $n\in \mathbb{N}$.\\
		$ii)$ The point $w=0$ is a local minima of $F_n$.
		\\
		
		Then, by standard properties of subdifferentials of convex functions and using (\ref{sd1}), we derive
		$$\begin{aligned}
		0 &\in \partial ( \Phi^{\circ}(u_n,\cdot)+\Psi(\cdot+u_n)-\Psi(u_n)+G_n)(0)\\
		&\subset \partial\Phi^{\circ}(u_n,\cdot)(0)+\partial \Psi(u_n)+\{G_n\}\\
		&= \partial \Phi(u_n)+\partial \Psi(u_n)+\{-w_n\}.
		\end{aligned}$$
		From this it follows that
		$$w_n \in \partial \Phi(u_n)+\partial \Psi(u_n), \,\,\,\forall n\in \mathbb{N}.$$
		The reciprocal is a direct consequence of the definition of $\partial \Phi(u_n)$ and $\partial \Psi(u_n)$.
	\end{proof}

\section{Abstract theorems involving the $G$-index theory}\label{Sec3}

	This section is devoted to establish our main results related with the existence of multiples critical points for functionals $I\in (H_1)$ that verify a condition of symmetry. More precisely, these results deal with the notions of the $G$-index theory introduced in \cite{ABS}. It is very important to mention that the result to be proved in this section generalize and complement several minimax type theorems involving multiplicity of critical points such as, among others, \cite[Theorems 3.8 and 3.9]{ABS} and \cite[Corollary 3.4]{Motreanu}. As in the previous section, $(X,\|\cdot \|) $ and $(X^*,\|\cdot \|_*)$ stand for a Banach space and its topological dual space, respectively.
	
	\subsection{Tools of the $G$-index theory} In the sequel, consider $G$ a topological compact group, with neutral element $e \in G$, that acts isometrically on $X$, i.e., there is is a continuous function
	$$
	\begin{array}{ccccl}
	\phi&: &G\times X&\rightarrow &X\\
	&&(g,v)&\mapsto&\phi(g,v)=gv
	\end{array}
	$$
	verifying
	\begin{itemize}
		\item[$(G_1)$] $ev = v,\,\, \forall v \in X$;
		\item[$(G_2)$] $(gh)v=g(hv),\,\, \forall v \in X,\, \forall g,h \in G$;
		\item[$(G_3)$]  \it{For each $g \in G$,
			\begin{equation*}
			\begin{array}{cccl}
			\phi_g:&X&\rightarrow &\,X \\
			&v &\mapsto &\phi_g(v)=gv
			\end{array}
			\end{equation*}
			is a linear isometric map, that is $\phi_g$ is linear and $\|gv\|=\|v\|,\,\, \forall v \in X,\,\,\forall g \in G $.}
	\end{itemize}	
	
	A subset $A$ of $X$ is said to be \textit{$G$-invariant} if $gA=A$ for every $g \in G$, where $gA\coloneqq \{gx \, :\, x\in A\}$. When $A \subset X$ is a $G$-invariant set, a map $\gamma: A\to X$ is said to be equivariant if $\gamma(gu)=g\gamma(u)$ for any $g \in G$, $u\in A$. We shall need of the following notation
	$$\Gamma_G(A)\coloneqq\{\gamma \in C(A,X) \;\; :\;\; \gamma \ \text{is equivariant}\}.$$ 
	
	A remarkable result of the measure theory assures that, under suitable conditions, it can be constructed on a group $G$ a \textit{left-invariant measure} $\mu$, i.e, the measure $\mu$ verifies
	$$\int_G f(g^{-1}y) d\mu = \int_G f(y)d\mu,\,\, \forall g \in G,$$ 
	for any integrable function $f:G \longrightarrow \mathbb{R}$. If, in addition, the gorup $G$ is compact, the measure $\mu$ can be chosen such that
	$$\mu(G)=1.$$
	This measure $\mu$ is denominated the \textit{Haar's measure} of $G$ and the integral related with $\mu$ is the so called \textit{Haar's integral}. Further details and a theoretical discussion of this subject can be found in \cite{Nachbin}. It is very important to point out that the Haar's integral can be extended for $X$-valued measurable functions (see, e.g., \cite[Appendix B]{Ismael} for an explicit construction of such integral). It is well known that, for each $\beta \in C(Gu,X),$ where $Gu\coloneqq \{gu \, : \, g\in G\}$, the map $\eta:X\rightarrow X$ given by
	\begin{equation}\label{Haar3}
	\eta(u)\coloneqq \int_Gg\beta(g^{-1}u)d\mu,\,\, u\in X,
	\end{equation}
	belongs to $\Gamma_G(X)$. This fact will be explored later on.
	
	Now, let us recall the notions of $G$-index theory demanded in our work. Besides of the condition $(G_1)-(G_3)$ above, assume that following additional condition holds.
	\\ \\
	$(G_0)$ For some finite dimensional vector space $Y$, fixed $k \in \mathbb{N}$, the group $G$ acts diagonally on $Y^k$, that is
	$$ gv = g(v_1,...,v_k) = (gv_1,...,gv_k), $$
	for every $v=(v_1,...,v_k) \in Y^k$ and each $ g \in G$.
	Moreover, the action of  $G$ on $Y$ is such that for each equivariant map $\gamma:\partial U \rightarrow Y^{k-1},$ where $k\geq 2$ and $U$ is a bounded $G$-invariant open set of $Y^k$ with $0 \in U,$ there is $u \in \partial U$ such that $\gamma(u) = 0$.
	\\
	
	Denote by $\Sigma$ the class of subsets of $(X-\{0\})$ that are $G$-invariant and closed in $X$. Let $Y$ be the vector space fixed in the condition $(G_0)$ above.
	
	\begin{definition}[$G$-index]
		The $G$-index of $A \in \Sigma\setminus\{\varnothing\}$ is defined as
		$$
		\gamma_G(A)\coloneqq\min\big\{k\in \mathbb{N}\setminus\{0\} \; \; : \;\; \exists \phi: A \rightarrow Y^k\setminus\{0\},\,\phi \in \Gamma_G(A) \big\},
		$$
		if such integer exists. Otherwise, we set $\gamma_G(A)\coloneqq+\infty$. Finally, we also set $\gamma_G(\emptyset)\coloneqq0$.
	\end{definition}

	Note that when $G=\mathbb{Z}_2$ the $G$-index introduced above turns out the well known genus of symmetric subset of $(X-\{0\})$ (see e.g. \cite{Rabinowitz1} for related topics).

	In the next proposition, we list some properties of the $G$-index $\gamma_G$.
	\begin{proposition}\label{GP}
		For every $A,B \in \Sigma$ the following facts hold:
		\begin{itemize}
			\item[$i)$] If there exists $\phi:A \rightarrow B$, $\phi \in\Gamma_G(A)$, then $\gamma_G(A)\leq \gamma_G(B);$
			\item[$ii)$] $A \subset B$ implies that $\gamma_G(A)\leq\gamma_G(B);$
			\item[$iii)$] $\gamma_G(A\cup B)\leq \gamma_G(A)+\gamma_G(B);$
			\item[$iv)$] $\gamma_G(\overline{A\setminus B})\geq\gamma_G(A)-\gamma_G(B),$ provided $\gamma_G(B)<\infty;$
			\item[$v)$] If $A$ is a compact set, then $\gamma_G(A)<\infty$;
			\item[$vi)$] If $A$ is a compact set, then we have $$\gamma_G(N_\delta(A))=\gamma_G(A),$$ $\delta \approx 0^+,$ where $$N_\delta(A)\coloneqq\{x \in X:\, d(x,A)\leq \delta\}.$$
		\end{itemize}
	\end{proposition}
	\begin{proof}
		\cite[Proposition 2.15]{ABS}.
	\end{proof}

	\subsection{Minimax theorems related with the $G$-index theory} We are going to prove some minimax type theorems for functionals $I\in (H_1)$ related with the $G$-index theory. Preceding such theoretical results, let us recall an important deformation lemma valid for functionals $I\in (H_1)$. By a deformation, we understand a mapping of the form
	\begin{equation*}
	\alpha_s\coloneqq \alpha(s,\cdot):W\subset X \rightarrow X, \,\,s \in [0, s_0]
	\end{equation*}
	where $\alpha_0 \equiv \mathcal{I}d|_W$, with $\alpha \in C([0, s_0]\times W,X)$ and $\mathcal{I}d|_W$ denotes the restriction of the identity map $\mathcal{I}d$ on $X$ to $W$. From now on, we fix for a functional $I\in(H_1)$
	$$I^{c}\coloneqq I^{-1}((-\infty, c])\, \mbox{ for every }\, c \in \mathbb{R},$$
	$$K\coloneqq \{u \in X: \, u \,\, \text{is a critical point of}\,\, I\},$$
	and
	$$K_c\coloneqq  \{u \in K:\, I(u)=c\}.$$
	Now, we recall an important deformation lemma valid for functionals of the class $(H_1)$ due to Motreanu-Panagiatopoulos \cite[Theorem 3.1]{Motreanu}.
	\begin{lemma}\label{Def}
		Suppose that $I=\Phi + \Psi\in (H_1)$ satisfies the $(\rm PS)_c$ condition for some $c\in\mathbb{R}$ and let $N$ be a neighbourhood of $K_c$. Then, fixed $\varepsilon_0 > 0$, there exists $\varepsilon \in (0,\varepsilon_0)$ such that, for each compact set $A \subset X\setminus N$ with
		\begin{equation}\label{Compact}
		c \leq \displaystyle{\sup_{u \in A}} I(u) \leq c+\varepsilon,
		\end{equation}
		there exist a closed set $W,$ with $A\subset \text{int}({W}),$ and a deformation $\alpha_s: W \rightarrow X,$ with $0 \leq s \leq s_0 \approx 0^+,$ such that
		\begin{itemize}
			\item[$i)$] $ \|\alpha_s(u) - u \| \leq s,\,\,\, \forall u \in W;$
			\item[$ii)$] $I(\alpha_s(u))-I(u) \leq Ms,\,\,\, \forall u \in W,\,s\in[0,s_0]$, for some positive $M>0$ independent of $u$ and $s>0$;
			\item[$iii)$] $I(\alpha_s(u)) - I(u) \leq -2\varepsilon s, \,\,\, \forall u \in W,\,\, I(u)\geq c-\varepsilon;$
			\item[$iv)$] $\displaystyle{\sup_{u \in A}}I(\alpha_s(u)) - \displaystyle{\sup_{u \in A}} I(u)\leq -2\varepsilon s,\, \forall s\in[0,s_0]$;
		\end{itemize}
	\end{lemma}
	\begin{proof} 
		The proof is well made in \cite[Theorem 3.1]{Motreanu}.
	\end{proof}

	Borrowing the ideas in \cite{ABS}, we introduce the following concept.
	\begin{definition}
		Let $\Psi:X\rightarrow (-\infty, +\infty]$ be an arbitrary functional on $X$ and let $G$ be a compact topological group that acts on $X$. We say that $\Psi$ is compatible with the action of $G$ on $X$ $($briefly $G$-compatible$)$ if the following inequality holds
		\begin{equation}\label{Comp}
		\Psi\left(\int_G g^{-1}\beta(g u) d\mu\right) \leq \int_G\Psi(g^{-1}\beta(g u))d\mu,
		\end{equation}
		for every fixed $u\in X$, $\beta \in C(Gu,X),$ where $Gu\coloneqq \{gu \, : \, g\in G\}$ and $\mu$ denotes the normalized Haar measure on $G$.
	\end{definition}	
	The technical property pointed out below plays an important role in the proof of the main abstract results of this section.

	\begin{lemma}
	Let $\Phi \in {\rm Lip_{loc}}(X,\mathbb{R})$. Then $\Phi^{\circ}(u,\cdot)$ is a $G$-compatible function for each fixed $u\in X$, i.e., fixed $u\in X$, the function $\Phi^{\circ}(u,\cdot)$ verifies (\ref{Comp}).
	\end{lemma}
	\begin{proof}
		The result is a natural consequence of Proposition \ref{21}-$i)$.
	\end{proof}

	The lemma below is an equivariant version of Lemma \ref{Def} and generalizes the results in \cite{Szulkin} and \cite{ABS}.
	
	\begin{lemma}\label{EDL}
		Suppose that $I=\Phi + \Psi\in (H_1)$ satisfies the $(\rm PS)_c$ condition. Assume that $\Phi$ and $\Psi$ are $G$-invariant functionals and $\Psi$ is a $G$-compatible functional. Moreover, suppose that $G$ acts isometrically on $X$. Under the hypothesis of Lemma \ref{Def}, the same conclusions hold with $\alpha_s: W \rightarrow X$  an equivariant mapping for each fixed $s\in [0,s_0]$, whenever $A$ in (\ref{Compact}) is assumed a $G$-invariant set.
	\end{lemma}	

	\begin{proof}
		Let $\alpha_s$ be the map given in Lemma \ref{Def} and set
		\begin{equation*}
		\widetilde{\alpha}_s(u)\coloneqq \int_Gg^{-1}\alpha_s(gu)d\mu.
		\end{equation*}
		A carefully reading of the proof of Theorem 3.1 in \cite{Motreanu} allows us to derive that, in view of the preceding lemma, the same estimates obtained there still holds for $\widetilde{\alpha}_s$ as given above (see also \cite[Remark 3.2]{Motreanu} for a related comment).
	\end{proof}

	We recall below the famous variational Ekeland's principle \cite[Theorem 1]{Ekeland} that will be an useful fact in our study.
	
	\begin{theorem}[Variational Ekeland's principle]\label{PVE}
		Let $(Y,d)$ be a complete metric space. Suppose that $\varphi:Y\rightarrow(-\infty,\infty]$ is a proper lower semicontinuous functional bounded from below. Given $\delta$, $\tau>0$ and $u_0 \in Y$ such that
		\begin{equation}
		\displaystyle{\inf_{u\in Y}}\varphi(u) \leq \varphi(u_0)\leq\displaystyle{\inf_{u\in Y}}\varphi(u)+\delta,
		\end{equation}
		then, there exists $v_0 \in Y$ verifying
		\begin{itemize}
			\item [$i)$]$\varphi(v_0)\leq\varphi(u_0),\,\,\,d(v_0,u_0)\leq {1}/{\tau};$
			\item [$ii)$]$\varphi(v)-\varphi(v_0)\geq-\delta\tau d(v,v_0),\, \forall v \in Y$.
		\end{itemize}
	\end{theorem}

	In the sequel we prove some minimax type results for functionals $I \in (H_1)$.
	Our first result generalizes \cite[Theorem 4.4]{Szulkin} and complement \cite[Corollary 3.6]{Motreanu}. 
	
	\begin{theorem}\label{SP1}
		Let $I=\Phi + \Psi\in (H_1)$ be such that $I(0)=0$. Assume that $\Phi$ and $\Psi$ are $G$-invariant functionals with $\Psi$ a $G$-compatible functional. Suppose that there exist subspaces $Y,Z$ of $X$ such that $X=Y\oplus Z,$ $\dim Y<\infty,$ $Z$ is closed and
		\begin{itemize}
			\item[$i)$] There are numbers $r,\tau>0$ such that $ I|_{\partial B_r(0)\cap Z} \geq \tau;$
			\item[$ii)$] There exists $V$ a finite dimensional subspace of $X$ such that $I(u) \rightarrow -\infty$ as $\|u\|\rightarrow \infty$ with $u \in V$ and ${\rm dim}\, V> {\rm dim}\, Y$.
		\end{itemize}	
		Then, setting $k\coloneqq ({\rm dim}\, V-{\rm dim}\, Y)$, it is possible to define levels of the form
		$$ c_j\coloneqq \inf_{A \in \Lambda_j}\sup_{u \in A}I(u),\quad j\in\{1,...,k\}$$
		in a such way that, if $I$ satisfies the $(\rm PS)_{c_j}$ condition, then $I$ has at least $k$ critical points $\{u_1,...,u_k\}$ with $I(u_j)=c_j$.
	\end{theorem}
	Next, let us introduce some notations and facts required for a precise description of levels $c_j$. Set $M_0:=\overline{B}_{R_0}(0)\cap V$ with $R_0>r$ and $I|_{\partial M_0}\leq -c_0$, for some fixed $c_0>0$ (this construction is possible in view of the assumption $ii)$ above). Following the assumptions in Theorem \ref{SP1}, consider $m:={\rm dim} Y$, ${\rm dim}V= m+k$ and fix the following notations
	$$\mathcal{F}:=\left\{\eta\in \Gamma_G(M_0):\,\,\begin{aligned}\text{there exists}\, &d=d(\eta)>0,\,\,\, \eta|_{\partial M_0}\approx \mathcal{I}d|_{\partial M_0}\,\,\text{in}\,\,I^{-d}\subset X\setminus\{0\}\\
	&\,\,\quad \quad \text{by an equivariant homotopy.}\end{aligned}\right\}.
	$$
	For each $j \in \{1,...,m+k\}$ set
	$$\tilde{\Lambda}_j:=\left\{\eta(M_0\setminus U)\,\,:\,\, \begin{aligned}\eta\in\mathcal{F}&,\, U\,\text{is $G$-invariant and open in}\, M_0,\,U\cap\partial M_0=\emptyset,\\
	& \text{with}\, \gamma_G(W)\leq (m+k)-j,\, \text{for}\,W\in\Sigma,\,W\subset U.\end{aligned}\right\}$$
	and 
	$$\Lambda_j:=\{A\subset X: A\,\text{is compact, $G$-invariant and for each open}\, U\supset A,\,\text{there is}\, A_0\in\tilde{\Lambda}_j,\,A_0\subset U\}.
	$$
	Finally, we fix
	$$c_j:=\inf_{A \in \Lambda_j}\sup_{u \in A}I(u).$$
	The collections $\Lambda_j$ and levels $c_j$ verifies the following properties
	\begin{itemize}
		\item[$(P_1)$] $(\Lambda_j,d_H)$ is a complete metric space, with $d_H$ denotes the \textit{symmetric Hausdorff} distance, namely, 
		$$d_H(A,B):=\max\{\sup_{x\in A}d(x,B),\sup_{y\in B}d(y,A)\},\quad A,\,B\in \Lambda_j.$$
		\item[$(P_2)$] $c_j \geq \tau$, for all $j > m;$
		\item[$(P_3)$] $\Lambda_{j+1}\subset \Lambda_j$$;$
		\item[$(P_4)$] Let $A \in \Lambda_j$ and $W$ be a closed $G$-invariant set containing $A$ in its interior. Assume also that $\alpha:W\rightarrow X$ is an equivariant mapping such that 
		$$\alpha|_{W \cap I^{-d}}\approx \mathcal{I}d|_{W \cap I^{-d}}\,\,\,\text{in}\,\,\,I^{-d},$$ 
		for some $d>0$, by an equivariant homotopy, then $\alpha(A) \in \Lambda_j$;
		\item[$(P_5)$] For each compact $B$ with $B \in \Sigma,$ $\gamma_G(B)\leq p,$ $I|_B>0,$ there exists a number $\delta_0>0$ such that $A\setminus\text{int}(N_\delta(B))\in \Lambda_j,$ for $A \in \Lambda_{j+p},$ $\delta\in(0,\delta_0)$.
	\end{itemize}
	The properties $(P_1)-(P_5)$ follow as made in \cite[Lemma 3.10]{ABS} and \cite[Theorem 4.4]{Szulkin}. Now, we are able to proof the Theorem \ref{SP1}.
	\begin{proof}[Proof of Theorem \ref{SP1}]
		From $(P_2)-(P_3)$ above we derive that
		$$0<\tau\leq c_{m+1}\leq\cdots\leq c_{m+k}.$$
	Let us perform to show that $\gamma_G(K_{c_j})\neq 0$ (and so $K_{c_j}\neq \emptyset$), for $j\in\{m+1,...,m+k\}$. Assume that it occurs $c:=c_j=\cdots=c_{j+p}$, for some $p\geq 0$. The proof is over if we get $\gamma_G(K_c)\geq p+1$. Indeed, it assures that $\gamma_G(K_{c_j})\geq 1$ for any $j$. Arguing by contradiction, suppose that $\gamma_G(K_c)<p+1$. One can see that $K_c$ is a compact $G$-invariant set with $0 \notin K_c$. For Lemma \ref{GP}-$v)$, we can fix $\delta\approx 0^+$ such that $\gamma_G(N_{2\delta}(K_c))=\gamma(K_c)$. Set $N=N_\delta(K_c)$ and $\varepsilon_0=\min{1,\delta}$. Now, choose $\varepsilon\in(0,\varepsilon_0)$ as in Lemma \ref{Def} and define
		$$\begin{array}{cccl}
		\varphi:&\Lambda_{j}&\rightarrow& (-\infty,\infty]\\
		& A&\mapsto& \varphi(A):={\displaystyle \sup_{u \in A}}I(u).
		\end{array}
		$$
	It can quickly seen that $\varphi$ is a l.s.c. functional  bounded from below by $c$ on the complete metric space $\Lambda_j$. Take $A_1\in \Lambda_{j+p}\subset \Lambda_j$ satisfying $\varphi(A_1)<c+\varepsilon^2=c_{j+p}+\varepsilon^2$ and let $A_2:=A_1\setminus{\rm int}(N_{2\delta}(K_c))$. By the Property $(P_5)$ above, for $\delta \approx 0^+$, it holds $A_2\in \Lambda_j$. By applying the Theorem \ref{PVE} to the functional $\varphi$, we can find $A\in \Lambda_j$ satisfying
	$$\varphi(A)\leq \varphi(A_2)<c+\varepsilon^2,\quad d_H(A,A_2)\leq \varepsilon$$
	and
	\begin{equation}\label{E1}
	\varphi(B)-\varphi(A)\geq -\varepsilon d_H(B,A),\quad B\in \Lambda_j.
	\end{equation}
	The choosing of $\varepsilon$ yields $N\cap A=\emptyset$. Furthermore, we have 
	$$c\leq \sup_{u\in A}I(u)=\varphi(A)<c+\varepsilon^2<c+\varepsilon.$$
	Therefore, we can apply the Lemma \ref{EDL} with $A$ and $N=N_{\delta}(K_c)$ to obtain an equivariant deformation $\alpha_s$, $0\leq s\leq s_0$, $s_0\approx 0^+$, satisfying $i)-iv)$ of Lemma \ref{Def}. By setting $B=\alpha_s(A)$, the condition in $ii)$ of Lemma \ref{Def} implies $I(\alpha_s(u))\leq I(u)+Ms_0$, $u\in A$. Then $I(\alpha_s(u))\leq -d+Ms_0$, provided $I(u)<-d$. Since $s_0\approx0^+$, we derive that $-\tilde{d}=-d+Ms_0<0$ and $\alpha_s(u)\in I^{-\tilde{d}}$. Hence, $\alpha_s$ can be viewed as an equivariant homotopy satisfying $(P_4)$ and consequently $B=\alpha_s(A)\in \Lambda_j$. By replacing $B$ in (\ref{E1}) and using the Part $iv)$ of Lemma \ref{Def} we find
	$$-2\varepsilon s\geq \sup_{u\in A} I(\alpha_s(u))-\sup_{u\in A}I(u)=\varphi(B)-\varphi(A)\geq -\varepsilon s.$$
	This contradiction concludes the proof.
	\end{proof} 
	
	The next corollary extends the information given in \cite[Corollary 3.6]{Motreanu} and works like a generalized version of symmetric Mountain Pass Theorem for functionals $I\in (H_1)$.
	\begin{corollary}\label{SP2}
		Let $I=\Phi + \Psi \in (H_1)$ be such that the $(\rm PS)$ condition holds and $I(0)=0$. Assume that $\Phi$ and $\Psi$ are $G$-invariant functionals with $\Psi$ a $G$-compatible functional. Finally, assume that there exist subspaces $Y,Z$ of $X$ such that $X=Y\oplus Z,$ $\dim Y<\infty,$ $Z$ is closed and
		\begin{enumerate}
			\item[$i)$] There are numbers $r,\tau>0$ such that $ I|_{\partial B_r(0)\cap Z} \geq \tau;$
			\item[$ii)$] For each positive integer $k$ there is a $k$-dimensional subspace $X_k$ of $X$ such that $I(u) \rightarrow -\infty$ as $\|u\|\rightarrow \infty$ with $u \in X_k$.
		\end{enumerate}	
		Then $I$ has a sequence critical values $(c_j)$ with $c_j\rightarrow \infty$ as $j\rightarrow \infty$.
	\end{corollary}
	The proof of corollary requires the result in the lemma below, which is a generalized version of Lemma 3.7 in \cite{ABS}.
	\begin{lemma}\label{TL}
		Suppose that $I=\Phi + \Psi\in (H_1)$ satisfies the $(\rm PS)$ condition. Moreover, let $(c_j)$ be a real sequence such that $c_j \rightarrow c \in \mathbb{R}$. Then, given $\delta>0,$ there exists $j_0 \in \mathbb{N}$ such that
		$$ K_{c_j}\subset N_\delta(K_c),\quad j\geq j_0$$
	\end{lemma}
	
	\begin{proof}
	If the result do not hold, it would be possible to find a subsequence $(c_{j_k})$ of $(c_j)$, a number $\delta_0>0$, and a sequence $(u_k)$ with $u_k \in K_{c_{j_k}}$ such that
	\begin{equation}\label{DI}
	d(u_k,K_c)>\delta_0,\,\,\, \forall k \in \mathbb{N}.
	\end{equation}
	On account that $u_k\in K_{c_{j_k}}$, one gets
	\begin{equation}\label{CPk}
	\Phi^\circ(u_k,v-u_k)+\Psi(v)-\Psi(u_k) \geq 0,\quad \forall v \in X,
	\end{equation}
	as well as
	$$
	I(u_k)=c_{j_k}\rightarrow c,
	$$
	so that $(u_k)$ is a $(\rm PS)_c$ sequence for the functional $I$. The $(\rm PS)$ condition yields the existence of $u_0 \in X$ and a subsequence of $(u_k)$, still denoted by $(u_k)$, such that $u_k \rightarrow u_0 \quad \mbox{in} \quad X$. By taking $v=u_0$ in (\ref{CPk}), we get $\limsup \Psi(u_k) \leq \Psi(u_0)$. This inequality in addition to the semicontinuity property of $\Psi$ and $\Phi^\circ(\cdot,\cdot)$ gives $\lim \Psi(u_k)=\Psi(u_0)$, so that $u_0 \in K_c.$ Hence $d(u_k,K_c) \to 0$ as $k \to \infty$, against (\ref{DI}).
	\end{proof}
	Next, we prove the Corollary \ref{SP2}.
	\begin{proof}[Proof of Corollary \ref{SP2}]
		By arguing as in the Theorem \ref{SP1} for each $k$-dimensional vector space $X_k$, $k>{\rm dim} Y$, we obtain a list of critical levels
		$$0<\tau\leq c_1\leq c_2\leq \cdots\leq c_j\leq \cdots .$$
		It remains to show that $c_j \to \infty$. If $c_j\nrightarrow \infty$, then there exists $c>0$ such that $c_j\to c$. Since $K_{c_j}\neq \emptyset$, by taking $u_j\in K_{c_j}$ and accounting $c_j\to c$, we find a $(\rm PS)_c$ sequence for $I$. Arguing as in the last lemma, one can sees that $K_c\neq \emptyset$. In addition, the conditions on $I$ imply that $K_c$ is a compact and $G$-invariant set with $0\notin K_c$.
		
		Now, the proof follows similar to what is made in \cite[Theorem 3.9]{ABS}. Fix $\delta\approx 0^+$ such that $\gamma_G(N_{2\delta}(K_c))=\gamma_G(K_c)=:r\in\mathbb{N}$ and set
		$\varepsilon_0:=\min\{1,\,\delta\},$
		take $\varepsilon \in (0,\varepsilon_0)$ as in Lemma \ref{Def} and define, for each $j\in \mathbb{N}$,
		
		$$\begin{array}{cccl}
		\varphi_j:&\Lambda_j&\rightarrow&(-\infty,+\infty]\\
		& A&\mapsto& \varphi(A)\coloneqq{\displaystyle\sup_{u \in A}I(u).}
		\end{array}
		$$
		Clearly $\varphi_j$ is a l.s.c. functional that is bounded from below. Let $A_1 \in \Lambda_{j+r}$ be such that
		$$\varphi_{j+r}(A_1)<c_{j+r}+\frac{\varepsilon^2}{2}.$$
		Since $(c_j)$ is convergent sequence, for large $ j$,
		$$\varphi_{j+r}(A_1)<c_{j+r}<c_j+\varepsilon^2.$$
		Fix $j\approx \infty$ satisfying the last inequality. Choosing $A_2:=\overline{A_1\setminus\text{int}(N_{2\delta}(K_c))}$, from $(P_5)$ above, we have $A_2 \in \Lambda_{j}$ and $\varphi_{j}(A_2)\leq\varphi_{j}(A_1)=\varphi_{j+r}(A_1)$ (note that $A_1\in \Lambda_{j+r}\subset\Lambda_j$). Setting $N:=N_\delta(K_c)$, the Lemma \ref{TL} implies that $K_{c_{j}}\subset N$ if $j$ is large enough. Hence, $N$ can be viewed as neighborhood of $K_{c_j}$. Now, the Theorem \ref{PVE} and the Lemmas \ref{Def} and \ref{EDL} can be used to achieve a contradiction, arguing with $\varphi_j$ instead of $\varphi$ as in the final of Theorem \ref{SP1}.  
		
	\end{proof}
	
\section{Applications to quasilinear problems involving the 1-Laplacian operator}\label{Sec4}

In this section, we will apply the abstract results from the previous section to prove the existence of infinite solutions to certain classes of quasilinear problems involving the $1$-Laplacian operator with different models of discontinuous nonlinearities.

\subsection{A problem with a discontinuous nonlinearity with linear-subcritical growth}\label{Ap1}

Here we will perform on the following class of problem whose the nonlinearity contains multiple discontinuities. Specifically, we study the problem
	$$
	\left \{
	\begin{array}{rclcl}
		-\Delta_1 u &=& \lambda\mbox{sign}(u) + H(|u|-a) |u|^{q-2}u, &\mbox{in}& \Omega;\\
		u &=& 0, & \mbox{on}  & \partial\Omega, \\
	\end{array}
	\right.\eqno{(P_{\lambda,a})}
	$$
where $\Omega \subset \mathbb{R}^N$ is a smooth bounded domain with $N \ge2$, $\lambda,a>0$ and $q \in (1,1^*)$. The notation $\mbox{sign}(s)$ denotes the sign function, defined as
	$$\mbox{sign}(s) \coloneqq\left \{
	\begin{array}{rclcl}
		\dfrac{s}{|s|},& s\neq 0;\\
		0,& s=0,&
	\end{array}
	\right.
	$$
and $H$ represents the Heaviside function, given by
	\begin{equation}\label{Heaviside}
		H(t) \coloneqq\left \{
		\begin{array}{rclcl}
			0,& t < 0;\\
			1,& t \ge 0.&
		\end{array}
		\right.
	\end{equation}
In order to provide a variational framework, consider the function $g_a : \mathbb{R} \to \mathbb{R}$, defined for a fixed $a>0$ by
	$$g_a(s) = H(|s| - a) |s|^{q-2}s,$$
which is locally bounded in $\mathbb{R}$. Its primitive, 
	$$G_a(s) \coloneqq \int_{0}^{s} H(|t| - a) |t|^{q-2}t dt,$$
is locally Lipschitz in $\mathbb{R}$. In this case, defining
	$$\underline{g}_a(s) \coloneqq \lim_{\varepsilon \to 0^+} ess \inf \{g_a(t)\; : \; \vert s-t\vert  < \varepsilon \} = \left\{ \begin{array}{rlc}
		|s|^{q-2}s,& \mbox{if} \ |s|> a;& \\
		0,& \mbox{if} \ |s| < a;& \\
		0,& \mbox{if} \ s= a;& \\
		-a^{q-1},& \mbox{if} \ s = -a,&
	\end{array}
	\right.$$
and
	$$\overline{g}_a(s) \coloneqq \lim_{\varepsilon \to 0^+} ess \inf \{g_a(t)\; : \; \vert s-t\vert  < \varepsilon \} = \left\{ \begin{array}{rlc}
		|s|^{q-2}s,& \mbox{if} \ |s|> a;& \\
		0,& \mbox{if} \ |s| < a;& \\
		a^{q-1},& \mbox{if} \ s= a;& \\
		0,& \mbox{if} \ s = -a,&
	\end{array}
	\right.$$
by the seminar due to Chang \cite{C}, it is possible to conclude that
	\begin{equation*}
		\partial G_a(s) = \left[\underline{g}_a(s), \overline{g}_a(s)\right] =	\left\{ \begin{array}{rlc}
			\{|s|^{q-2}s\},& \mbox{if} \; |s| >a;& \\
			\{0\},& \mbox{if} \; |s| <a;&\\
			\left[0,a^{q-1}\right],& \mbox{if} \; s=a;& \\
			\left[-a^{q-1},0\right],& \mbox{if} \; s=-a.&
		\end{array}
		\right.
	\end{equation*}
For future reference, note that, by definition of $G_a$,
	\begin{equation}\label{growthgradg}
		|t| \le |s|^{q-1}, \; \; \mbox{for any} \ t \in \partial G_a(s) \subset \mathbb{R}.
	\end{equation}
Furthermore, direct computations yield
	\begin{equation}\label{growthGs>a}
		G_a(s) = \left\{ \begin{array}{rlc}
			\dfrac{|s|^q}{q} - \dfrac{a^q}{q},& \mbox{if} \; |s| \ge a;& \\
			0,& \mbox{if} \; |s| <a,&
		\end{array} \right.
	\end{equation}
as well as,
	\begin{equation}\label{growthG}
		|G_a(s)| \le \dfrac{1}{q}|s|^q, \; \; \mbox{for any} \ s \in \mathbb{R}.
	\end{equation}	
	
Similarly, consider the function
	$$F(s) \coloneqq \int_{0}^{s}\mbox{sign}(t)dt,$$
which is locally Lipschitz in $\mathbb{R}$. Its subdifferential is
	$$\partial F(s) = \left\{ \begin{array}{rlc}
	\mbox{sign}(s),& \mbox{if} \; |s| \neq 0;& \\
	\left[-1,1\right],& \mbox{if} \; s=0.& 
	\end{array}
	\right.$$

Hereafter, we shall denote by $\mathcal{M}(\Omega,\mathbb{R}^N)$ (shortly $\mathcal{M}(\Omega)$) the space of all Radon measure on $\Omega$ and by $BV(\Omega)$ the space of the functions of bounded variation, that is
	\begin{equation}\label{BVdef}
		BV(\Omega) \coloneqq \{u \in L^1(\Omega) \; \; : \;\; Du \in \mathcal{M}(\Omega)\},
	\end{equation}
where $Du$ designates the distributional derivative of $u \in L^1(\Omega)$. Let us recall that $u \in BV(\Omega)$ if, and only if, $u\in L^1(\Omega)$ and
	$$\int_\Omega |Du| \coloneqq \sup \left\{\int_\Omega u\ \mbox{div}\phi \ dx \; \; : \;\; \phi \in C_0^1(\Omega,\mathbb{R}^N), \ \|\phi\|_\infty \le 1\right\} < \infty.$$
It is well known that a trace operator of the form $BV(\Omega)\hookrightarrow L^1(\partial \Omega)$ is well defined, so that the space $BV(\Omega)$ endowed by the norm
	$$\|u\|_{BV(\Omega)} \coloneqq \int_\Omega |Du| + \int_{\partial \Omega} |u| d \mathcal{H}^{N-1},$$
with $\mathcal{H}^{N-1}$ denoting the $(N-1)$-dimensional Hausdorff measure, is a Banach space. Moreover, the embedding
	$$BV(\Omega) \hookrightarrow L^r(\Omega), \quad r \in [1,1^*],$$
is continuous and compact for $r \in [1,1^*)$. For more details about the space $BV(\Omega)$ we refer \cite{AFP,EG}.

In the following definition, by using the notions introduced by Kawohl and Schuricht \cite{KS}, Degiovanni and Magroni \cite{DM} and Pimenta, dos Santos and Júnior \cite{PSJ}, we give a precise description for a solution of $(P_{\lambda,a})$ as follows:

\begin{definition}\label{def1}
	A solution for $(P_{\lambda,a})$ is a function $u \in BV(\Omega)$ such that, there exists ${\bf z} \in L^\infty(\Omega,\mathbb{R}^N)$ with $\|{\bf z}\|_\infty \le 1$, satisfying
		$$\left\{
			\begin{array}{rclcl}\displaystyle
				-\int_\Omega u \ \mbox{div}\, {\bf z}\ dx &=&  \displaystyle \int_\Omega |Du| + \int_{\partial \Omega} |u| d \mathcal{H}^{N-1} \\
				-\mbox{div} \, {\bf z} &=& \lambda\rho_1 (x) + \rho_2(x), \ \mbox{a.e. in} \ \Omega, 
			\end{array}
		\right.$$
	with $\rho_1(x) \in  \partial F(u)$ and $\rho_2(x) \in \partial G_a(u)$ almost everywhere in $\Omega$.
\end{definition}

For any $\lambda,a \ge 0$, we will consider the functionals $\Phi_1^\lambda, \Phi_2^a$ and $\Psi$ defined by 
	$$\Phi_1^\lambda (u) = \lambda \int_\Omega F(u)dx, \quad \Phi_2^a(u) = \int_\Omega G_a(u)dx, \quad u \in L^q(\Omega),$$
and
	$$\Psi(u)\coloneqq	\left\{ \begin{array}{clc} \displaystyle
		\int_\Omega |Du| + \int_{\partial\Omega} |u|d\mathcal{H}^{N-1},& \mbox{if} \; u \in BV(\Omega);& \\
		+\infty,& \mbox{if} \; u \in L^{q}(\Omega) \setminus BV(\Omega).&
	\end{array}
	\right.$$ 
Furthermore, our energy functional is $I_{\lambda,a}:L^q(\Omega) \to (-\infty,\infty]$ defined by
	$$I_{\lambda,a}(u) = \Phi_{\lambda,a}(u) + \Psi(u),$$
where $\Phi_{\lambda,a}(u) = \Phi_1^\lambda(u) + \Phi_2^a(u)$.

The conditions on $F$ and $G_a$ imply that $\Phi_{\lambda,a} \in Lip_{loc}(L^q(\Omega),\mathbb{R})$ and it is quickly seen that $\Psi$ is a lower semicontinuous and convex function. Thus, we may conclude that $I_{\lambda,a}\in (H_1)$, with $D(I_{\lambda,a}) = D(\Psi)=BV(\Omega)$. We would like to mention that, for a fixed $u \in L^p(\Omega)$, the subdifferential of $\Psi$ at $u$ can be identified as a subset of $L^p(\Omega)$, where $p=q/(q-1)$ denotes the conjugate exponent of $q$.

We present below some results related with the subdifferential of $\Psi$ required in our study. Their proofs can be found in \cite{ABS} Lemmas 4.23 and 4.24, respectively.

\begin{lemma}\label{Linf}
	If $\partial\Psi(u) \neq \emptyset$ and $u \in BV(\Omega)$, then $u \in L^\infty(\Omega)$.
\end{lemma}

\begin{lemma}\label{realiza}
	If $u \in BV(\Omega)$, then for each $w \in \partial\Psi(u),$ there exists ${\bf z} \in L^\infty(\Omega,\mathbb{R}^N)$, with $\|{\bf z}\|_\infty \le 1$, satisfying
		$$\left\{
		\begin{array}{rclcl}\displaystyle
			-\int_\Omega u \, \mbox{div}\ {\bf z}\ dx &=&  \displaystyle \int_\Omega |Du| + \int_{\partial \Omega} |u| d \mathcal{H}^{N-1}; \\
			-\mbox{div} \, {\bf z} &=& w, \ \mbox{a.e. in} \ \Omega. 
		\end{array}
		\right.$$
\end{lemma}

Next, we establish the connection between solutions of $(P_{\lambda,a})$ and critical points of $I_{\lambda,a}$.

\begin{lemma}\label{energyfunc}
	For any fixed $\lambda,a\ge0$, let $u_{\lambda,a}$ be a critical point of $I_{\lambda,a}$. Then, $u_{\lambda,a}$ is a solution of $(P_{\lambda,a})$ in the sense of the Definition \ref{def1}. Moreover, we have $u_{\lambda,a}\in BV(\Omega)\cap L^\infty(\Omega)$. 
\end{lemma}
\begin{proof}
	Consider $u_{\lambda,a} \in L^q(\Omega)$ a critical point of $I_{\lambda,a}$. Then, $I_{\lambda,a}(u_{\lambda,a})<\infty$ and
		$$0 \in \partial \Phi_{\lambda,a}(u_{\lambda,a}) +\partial \Psi(u_{\lambda,a}),$$
	which implies, for some $\xi_1 \in \partial\Phi_1^\lambda(u_{\lambda,a})$, $\xi_2 \in \partial\Phi_2^a(u_{\lambda,a})$ and $w \in \partial\Psi(u_{\lambda,a})$, that
		$$-(\xi_1+\xi_2) = w, \quad \mbox{a.e. in} \ \Omega.$$
	The characterizations of $\partial\Phi_1^\lambda(u_{\lambda,a})$ and $\partial\Phi_2^a(u_{\lambda,a})$ give us
		$$\lambda\rho_1(x) + \rho_2(x) = w(x), \quad \mbox{a.e. in} \ \Omega,$$
	with $\rho_1(x) \in \partial F(u_{\lambda,a})$ and $\rho_2(x) \in \partial G_a(u_{\lambda,a})$. The rest of the proof can be derived as consequence of Lemmas \ref{Linf} and \ref{realiza}.
\end{proof}

Our next step is proving that functional $I_{\lambda,a}$ satisfies the $(\rm PS)$ condition as follows:
\begin{lemma}\label{CondPS1}
	The functional $I_{\lambda,a}$ satisfies the $(\rm PS)$ condition for any $\lambda,a>0$ small enough.
\end{lemma}
\begin{proof}
	Suppose that $(u_n)$ is a $(\rm PS)$ sequence for $I_{\lambda,a}$. Thus, for some $c \in \mathbb{R}$
		$$I_{\lambda,a}(u_n) \to c, \quad \mbox{as} \ n \to \infty,$$
	and by Lemma \ref{caracPS} there exists $w_n \in \partial \Phi_{\lambda,a}(u_n) + \partial \Psi(u_n)$ with $w_n\to0$ in $L^p(\Omega)$. Since $\partial \Phi_{\lambda,a}(u_n) \subset \partial\Phi_1^\lambda(u_n) + \partial\Phi_2^a(u_n)$, it is possible to find $\rho_{1,n} \in \partial F(u_n)$ and $\rho_{2,n} \in \lambda\partial G_a(u_n)$ such that
		$$w_n(x) + \lambda\rho_{1,n}(x) + \rho_{2,n}(x) \in \partial \Psi(u_n).$$
	Thereby, by Lemma \ref{realiza},
		\begin{eqnarray}
			\Psi(u_n) &=& \int_\Omega |Du_n| + \int_{\partial \Omega} |u_n| d \mathcal{H}^{N-1} \nonumber \\
			&=& \lambda\int_\Omega \rho_{1,n} u_n dx + \int_\Omega \rho_{2,n} u_n dx + \int_\Omega w_nu_n dx, \quad \mbox{for any} \ n \in \mathbb{N}. \nonumber 
		\end{eqnarray}
	Setting $A(u_n) \coloneqq \Psi(u_n) - \left(\displaystyle \lambda\int_\Omega \rho_{1,n} u_n dx + \int_\Omega \rho_{2,n} u_n dx + \int_\Omega w_nu_n dx\right)$, by previous inequality we get $A(u_n)=0$ for any $n \in \mathbb{N}$. Thus, it follows that
		\begin{eqnarray}\label{lim1}
			o_n(1) + c &=& I_{\lambda,a}(u_n) - \dfrac{1}{q}A(u_n) \nonumber\\
			&=& \left(1-\dfrac{1}{q}\right)\|u_n\|_{BV(\Omega)} + \lambda\int_{\Omega} \left(\dfrac{1}{q} \rho_{1,n} u_n -  F(u_n)\right)dx + \int_{\Omega} \left(\dfrac{1}{q} \rho_{2,n} u_n - G_a(u_n)\right)dx \nonumber \\
			&& -\dfrac{1}{q} \int_{\Omega} w_n u_ndx.
		\end{eqnarray}
	Separately, note that, by combining the characterization of $\partial F(u_n)$ with the embedding $BV(\Omega) \hookrightarrow L^1(\Omega)$, we get
		\begin{equation}\label{lim2}
			\left|\lambda  \int_{\Omega} \left(\dfrac{1}{q} \rho_{1,n} u_n -F(u_n)\right)dx\right| \le \lambda \left(\dfrac{1}{q}\int_{\Omega} |\rho_{1,n}u_n|dx + \int_{\Omega} |u_n|dx\right) \le \lambda C \|u_n\|_{BV(\Omega)},
		\end{equation} 
	for some $C>0$. Note also that
		\begin{equation}\label{lim3}
			\dfrac{1}{q} \left|\int_{\Omega} w_n u_n dx\right|\le o_n(1)\|u_n\|_{BV(\Omega)}.
		\end{equation}
	By combining \eqref{lim1}, \eqref{lim2} and \eqref{lim3} we obtain for $\lambda>0$ small enough
		\begin{equation}\label{lim4}
			o_n(1) + c \ge C_1\|u_n\|_{BV(\Omega)} + \int_{\Omega} \left(\dfrac{1}{q} \rho_{2,n} u_n - G_a(u_n)\right)dx,
		\end{equation}
	for a convenient $C_1>0$ independent of $n \in \mathbb{N}$. Next, we claim that, for some $C_2>0$,
		\begin{equation}\label{lim5}
			\int_{\Omega} \left(\dfrac{1}{q} \rho_{2,n} u_n - G_a(u_n)\right)dx \ge -C_2.
		\end{equation}
	Indeed, first note that
		\begin{eqnarray}
			\int_{\Omega} \left(\dfrac{1}{q} \rho_{2,n} u_n - G_a(u_n)\right)dx \le \int_{[|u_n|>a]} \left(\dfrac{1}{q} \rho_{2,n} u_n - G_a(u_n)\right)dx + \int_{[|u_n|\le a]} \left(\dfrac{1}{q} \rho_{2,n} u_n - G_a(u_n)\right)dx \nonumber
		\end{eqnarray}
	Now, by \eqref{growthGs>a} and the definition of $\partial G_a(\cdot)$ we have
		\begin{eqnarray}\label{>a}
			\left|\int_{[|u_n|>a]} \left(\dfrac{1}{q} \rho_{2,n} u_n - G_a(u_n)\right)dx\right| =  \left|\int_{[|u_n|>a]} \left( \dfrac{1}{q}|u_n|^q - \dfrac{|u_n|^q}{q} - \dfrac{a^q}{q}\right)dx\right| \le \dfrac{a^q}{q}|\Omega|,
		\end{eqnarray}
	as well as,
		\begin{eqnarray}\label{<a}
			\left|\int_{[|u_n|\le a]} \left(\dfrac{1}{q} \rho_{2,n} u_n - G_a(u_n)\right)dx\right| =  \left|\int_{[|u_n|= a]} \left(\dfrac{1}{q} \rho_{2,n} u_n - G_a(u_n)\right)dx\right| \le \dfrac{2}{q}a^q|\Omega|.
		\end{eqnarray}
	Thus, \eqref{>a} and \eqref{<a} show \eqref{lim5} with $C_2 = \max \left\{\frac{a^q}{q}|\Omega|, \frac{2}{q}|\Omega|a^q\right\}$, as claimed. By \eqref{lim4} and \eqref{lim5}, it follows that
		$$o_n(1) + c \ge C_1\|u_n\|_{BV(\Omega)} - C_2.$$
	Hence, $(\|u_n\|_{BV(\Omega)})_{n \in \mathbb{N}}$ is a bounded sequence. Since $BV(\Omega) \hookrightarrow L^q(\Omega)$ compactly the sequence $(u_n)$ has a convergent subsequence in $L^q(\Omega)$ and the proof is complete.
\end{proof}

We are able to prove the first main result of this subsection, namely, the Theorem \ref{Main1}

\begin{proof}[Proof of Theorem \ref{Main1}]
	Firstly, by Lemma \ref{CondPS1}, if $\lambda \approx 0^+$, the functional $I_{\lambda,a}$ satisfies the $(\rm PS)$ condition for any $c \in \mathbb{R}$. Additionally, $I_{\lambda,a}$ is an even functional, meaning it is $\mathbb{Z}_2$-invariant. Therefore, to establish the existence of the sequence $\big(u_{\lambda,a}^{(n)}\big)_{n \in \mathbb{N}}$, it suffices to verify that $I_{\lambda,a}$ satisfies conditions $(i)$ and $(ii)$ of Corollary \ref{SP2}.

	\vspace{.5cm}
	\noindent {\it Verification of $(i)$:} It suffices to consider $u \in BV(\Omega) = D(I_{\lambda,a})$. Given $u \in BV(\Omega)$, by the embeddings $BV(\Omega) \hookrightarrow L^q(\Omega) \hookrightarrow L^1(\Omega)$ it holds 
		$$I_{\lambda,a}(u) \ge C_1\|u\|_q - \lambda C_2 \|u\|_q - C_3\|u\|_q^q,$$
	for some $C_1,C_2,C_3>0$. Now, for $\lambda \in (0,\lambda_0)$ with $\lambda_0$ sufficiently small, we can find a constant $D_1>0$ (depending on $\lambda$) such that
		$$I_{\lambda,a}(u) \ge D_1\|u\|_q - C_3\|u\|_q^q.$$
	Taking into account $q>1$, we derive that, for $\|u\|_q = r \approx 0^+$,
		\begin{equation}\label{Tau0}
		I_{\lambda,a}(u) \ge \tau,
		\end{equation}
	for a convenient $\tau>0$. We would like to emphasize the fact that $\tau>0$ can be chosen independent of $a$ (this fact will be used in the sequel). The above estimate proves the item $(i)$ of Corollary \ref{SP2} with $Z = L^q(\Omega)$.
	
	\vspace{.5cm}
	\noindent {\it Verification of $(ii)$:} Fix $k \in \mathbb{N}$ and let $X_k$ be a $k$-dimensional subspace of $C_0^\infty(\Omega) \subset L^q(\Omega)$. The equivalence of norms on $X_k$ allows us to find constants $A_k,B_k>0$ such that
		$$I_{\lambda,a}(u) \le A_k\|u\|_q + B_k\|u\|_q - \int_{\Omega} G_a(u)dx.$$
	By setting $C_k \coloneqq A_k+B_k$, the definition of $G_a$ and \eqref{growthGs>a} give us
		$$I_{\lambda,a}(u) \le C_k\|u^a\|q - \dfrac{1}{q} \|u^a\|_q^q + C_k\|u_a\|_q + \dfrac{a^q}{q}|\Omega|,$$
	where $u^a\coloneqq \chi_{[|u|> a]}u$ and $u_a\coloneqq \chi_{[|u|\le a]}u$, with $\chi$ denoting the characteristic function. Then, the last inequality implies that
		\begin{equation}\label{411}
			I_{\lambda,a}(u) \le C_k\|u^a\|_q - \dfrac{1}{q}\|u^a\|_q^q +C,
		\end{equation}
	where $C=C(\Omega,a,C_k)>0$. The next claim plays a crucial role in our argument.
	
	\vspace{.5cm}
	\noindent \underline{\it Claim 1:} Considering $(v_n) \subset L^q(\Omega)$ such that $\|v_n\|_q \to \infty$, then, for each fixed $a>0$, it holds
		$$\|v_n^a\|_q \to \infty.$$
	Indeed, note that, for each $n \in \mathbb{N}$, one can write
		$$\|v_n\|_q = \|(v_n)_a\|_q + \|v_n^a\|_q \le C(a,\Omega) + \|v_n^a\|_q,$$
	where $C(a,\Omega)>0$ is independent of $n$. Thus, since we have assumed $\|v_n\|_q \to \infty$ the statement in the claim holds, as desired.
	
	Thereby, from \eqref{411} we get
		$$I_{\lambda,a}(u) \to -\infty, \quad \mbox{as} \ \|u\|_q\to \infty,\,u\in X_k,$$
	which proves the item $(ii)$.
	
	With $(i)$ and $(ii)$ verified, Corollary \ref{SP2} ensures the existence of a sequence $(u_{\lambda,\,a}^n)_{n\in \mathbb{N}}$ of critical points of $I_{\lambda,a}$ for each $\lambda \approx0^+$ and $a>0$.

	It remains to prove $(M_1)$. In the sequel we may assume that $\lambda \in (0,\lambda_0)$, with $\lambda_0<1$. Suppose by contradiction that $(M_1)$ it is not verified. In this case, we get a sequence $a_k \to 0^+$ such that, for some critical point $u_{\lambda,a_k}^{n_k},$
		\begin{equation}\label{absurdo}
			\left|\,[ \,|u_{\lambda,a_k}^{n_k}| > a\,]\,\right|=0.
		\end{equation}
	To simplify notation, we write $v_{\lambda,k}$ instead of $u_{\lambda,a_k}^{n_k}$, that is $v_{\lambda,k}\coloneqq u_{\lambda,a_k}^{n_k}$. By recalling that $v_{\lambda,k}$ is a critical point of $I_{\lambda,a_k}$ obtained as a consequence of $(i)$ and $(ii)$ in Corollary \ref{SP2} we derive, from (\ref{Tau0}), that
		$$0<\tau \le I_{\lambda,a_k}(v_{\lambda,k}),$$
	with $\tau>0$ independent of $a_k>0$. On account that $v_{\lambda,k}$ is a critical point of $I_{\lambda,a_k}$, we get by Lemma \ref{realiza}  and by noting that
		$$\int_\Omega |Dv_{\lambda,k}| + \int_{\partial \Omega} |v_{\lambda,k}| d \mathcal{H}^{N-1} = \lambda \int_{\Omega} \rho_1^k v_{\lambda,k} dx + \int_{\Omega} \rho_2^k v_{\lambda,k} dx,$$
	with $\rho_1^k \in \partial F(v_{\lambda,k})$ and $\rho_2^k \in \partial G_{a_k}(v_{\lambda,k})$. Inasmuch as \eqref{absurdo} is in force, we derive
		\begin{eqnarray}
			0<\tau &\le& I_{\lambda,a_k}(v_{\lambda,k}) \nonumber \\
			&=& \lambda \int_{[|v_{\lambda,k}| \le a_k]} \rho_1^k v_{\lambda,k} dx + \int_{[|v_{\lambda,k}| \le a_k]} \rho_2^k v_{\lambda,k} dx \nonumber \\
			&& -\lambda \int_{[|v_{\lambda,k}| \le a_k]} F(v_{\lambda,k})dx - \int_{[|v_{\lambda,k}| \le a_k]} G_{a_k}(v_{\lambda,k})dx \nonumber \\
			&\le& \big(2\lambda a_k + C_q a_k^q\big)|\Omega| \nonumber \\
			&=& o_k(1), \nonumber
		\end{eqnarray}
	which lead us a contradiction, and the theorem follows.
\end{proof}

Now, we study the behavior of the solutions $u_{\lambda,a}^{(n)}$ of $(P_{\lambda,a})$, as $a \to 0^+$. Formally, as $a \to 0^+$, one should expect the solutions to converge to a solution of the following problem
$$
\left \{
\begin{array}{rclcl}
	-\Delta_1 u &=& \lambda\mbox{sign}(u) + |u|^{q-2}u, &\mbox{in}& \Omega;\\
	u &=& 0, & \mbox{on}  & \partial\Omega.\\
\end{array}
\right.\eqno{(P_{\lambda})}
$$
This is exactly what we prove in the following theorem.
\begin{theorem}\label{Main1.1}
	For $\lambda \in (0,\lambda_0)$, let $\Big(u_{\lambda,a}^{(n)}\Big)_{n \in \mathbb{N}}$ be as in Theorem \ref{Main1}. For $n \in \mathbb{N}$, there exists $u_{\lambda,0}^{(n)} \in BV(\Omega)$, such that, as $a \to 0^+$,
	$$
	u_{\lambda,a}^{(n)} \to u_{\lambda,0}^{(n)}, \quad \mbox{in $L^r(\Omega)$ and a.e. in $\Omega$,}
	$$
	for $r \in [1,1^*)$. Moreover, $u_{\lambda,0}^{(n)}$ is a bounded variation solution of $(P_{\lambda})$.
\end{theorem}

Before presenting the proof of the theorem above, let us remember that, according the notations in Section \ref{Sec3}, for each $j \in \mathbb{N}$, $c_{j,a}$, the minimax level associated to $I_a$, is given by
$$
c_{j,a}\coloneqq \inf_{A \in \Lambda_{j,a}}\sup_{u \in A}I_a(u),\quad j\in\{1,...,k\},
$$
where 
$$
\Lambda_{j,a} = \left\{A \subset X: 
\begin{array}{c}\mbox{$A$ is compact, $G-$invariant and for each open $U \supset A$,} \\ \mbox{there is $A_0 \in \tilde{\Lambda}_{j,a}$, $A_0 \subset U$.}
\end{array}\right\},
$$
$$
\tilde{\Lambda}_{j,a} = 
\left\{ \eta(M_k \backslash U) :
\begin{array}{c}
	\mbox{$\eta \in \mathcal{F}_a$, $U$ is $G$-invariant and open in $M_k$, $U \cap \partial M_k = \emptyset$},\\
	\mbox{with $\gamma_G(W) \leq k-j$, for $W \in \Sigma$, $W \subset U$, $k\geq j$}. 
\end{array}
\right\},
$$
for each $j \in \mathbb{N}$ and $k \geq j$. In the expression above, recall that
$$
\mathcal{F}_a = 
\left\{ 
\eta \in \Gamma_G(M_k) : 
\begin{array}{c}
	\mbox{there exists $d = d(\eta) > 0$, $\eta|_{\partial M_k} = \mathrm{Id}|_{\partial M_k}$ in $I_a^{-d} \subset X \backslash \{0\}$} \\
	\mbox{by an equivariant homotopy.}
\end{array} 
\right\},
$$
where $M_k = \overline{B}_{R_0}(0) \cup X_k$, $R_0$ is big enough and $I |_{\partial M_k} < 0$ and $X_k$ is a $k-$dimensional subspace of $X$ such that $I_a(u) \to -\infty$, as $\|u\| \to +\infty$, with $u \in X_k$.

\begin{lemma}
	\label{Lemma.monotonicity}
	If $0 < a_1 < a_2$, then $c_{j,a_1} \leq c_{j,a_2}$.
\end{lemma}
\begin{proof}
	First of all, by the definition of $I_{\lambda,a}$, note that 
	$$
	I_{\lambda,a_1}(u) \leq I_{\lambda,a_2}(u), \quad \forall u \in BV(\Omega).
	$$
	Moreover, if $X_k$ is a $k-$dimensional subspace of $X$ such that $I_{\lambda,a_2}(u) \to -\infty$, as $\|u\| \to +\infty$, with $u \in X_k$, then the same holds for $I_{\lambda,a_1}$.
	
	Now, let us prove that
	\begin{equation}
		\label{eq:F2inF1}
		\mathcal{F}_{a_2} \subset \mathcal{F}_{a_1}.
	\end{equation}
	Indeed, let $\eta \in \mathcal{F}_{a_2}$, then $\eta|_{\partial M_k}$ is homotopic to $\mathrm{Id}|_{\partial M_k}$ and there exists $d > 0$, such that
	$$
	I_{\lambda,a_2}|_{\partial M_k} \leq -d \quad \mbox{and} \quad I_{\lambda,a_2}(\eta(\cdot))|_{\partial M_k} \leq -d.
	$$
	Since $I_{\lambda,a_1} \leq I_{\lambda,a_2}$, it follows that
	$$
	I_{\lambda,a_1}|_{\partial M_k} \leq -d \quad \mbox{and} \quad I_{\lambda,a_1}(\eta(\cdot))|_{\partial M_k} \leq -d.
	$$
	Hence, $\eta \in \mathcal{F}_{a_1}$. As a result, $\Lambda_{j,a_2} \subset \Lambda_{j,a_1}$. Then
	\begin{eqnarray*}
		c_{j,a_1} & = & \inf_{A \in \Lambda_{j,a_1}} \sup_{u \in A} I_{\lambda,a_1}(u)\\
		& \leq & \inf_{A \in \Lambda_{j,a_1}} \sup_{u \in A} I_{\lambda,a_2}(u)\\
		& \leq & \inf_{A \in \Lambda_{j,a_2}} \sup_{u \in A} I_{\lambda,a_2}(u)\\
		& = & c_{j,a_2}.
	\end{eqnarray*}
	
\end{proof}

From now on, let us fix $a_0 > 0$. Moreover, to simplify the notation, we denote by $u_{\lambda,a}$ one of the solutions $u_{\lambda,a}^{(n)}$, for some $n \in \mathbb{N}$. Then, from the last result,
\begin{equation}
	\label{eq:Ibounded}
	I_{\lambda,a}(u_{\lambda,a}) \leq C, \quad \forall a \in (0,a_0).
\end{equation}

\begin{lemma}
	\label{lemma:bounded}
	For $\lambda > 0$ small enough, $(u_{\lambda,a})_{0 < a < a_0}$ is bounded in $BV(\Omega)$.
\end{lemma}
\begin{proof}
	Note that
	\begin{eqnarray*}
		C & \geq & I_{\lambda,a}(u_{\lambda,a}) - \frac{1}{q} \langle w_{\lambda,a} - \lambda \rho_{\lambda,a}^1 - \rho_{\lambda,a}^2, u_{\lambda,a}\rangle\\
		& = & \|u_{\lambda,a}\|_{BV(\Omega)} - \frac{1}{q}\langle w_{\lambda,a},u_{\lambda,a} \rangle + \frac{\lambda}{q} \langle \rho_{\lambda,a}^1,u_{\lambda,a} \rangle\\
		& & - \lambda \int_\Omega |u_{\lambda,a}| dx + \frac{1}{q} \langle \rho_{\lambda,a}^2,u_{\lambda,a}\rangle - \frac{1}{q} \int_\Omega G_a(u_{\lambda,a})dx\\
		& = & \left(1 - \frac{1}{q}\right)\|u_{\lambda,a}\|_{BV(\Omega)} - \lambda\left(1 - \frac{1}{q}\right)\|u_{\lambda,a}\|_q+ \int_\Omega \left(\frac{1}{q} \rho_{\lambda,a}^2 u_{\lambda,a} - G_a(u_{\lambda,a})dx \right)\\
		& \geq & \left(1 - \frac{1}{q}\right)\|u_{\lambda,a}\|_{BV(\Omega)} - \lambda\left(1 - \frac{1}{q}\right)\|u_{\lambda,a}\|_q\\
		& \geq & \left(1 - \lambda C \right)\left(1 - \frac{1}{q}\right)\|u_{\lambda,a}\|_{BV(\Omega)}.
	\end{eqnarray*}
	Hence, if $\lambda > 0$ is small enough, $(u_{\lambda,a})_{0 < a < a_0}$ is bounded in $BV(\Omega)$.
\end{proof}

By combining Lemma \ref{lemma:bounded} and the compact embedding $BV(\Omega) \hookrightarrow L^r(\Omega)$, with $1\le r <1^*$, we can consider $u_{\lambda,0} \in BV(\Omega)$, be such that 
\begin{equation}
	\label{eq:embedding}
	u_{\lambda,a} \to u_{\lambda,0}\quad \mbox{in $L^r(\Omega)$, for $1 \leq r < 1^*$.}
\end{equation}

As in \cite[Lemma 3.5]{PSJ}, we can show that there exists $\rho_{\lambda,0}^2 \in L^\frac{q}{q-1}(\Omega)$, such that, as $a \to 0^+$,
\begin{equation}
	\label{eq:embeddingrho1}
	\left\{\begin{array}{cl}
		\rho_{\lambda,a}^2 \rightharpoonup \rho_{\lambda,0}^2 & \mbox{in $L^\frac{q}{q-1}(\Omega)$;} \\
		\rho_{\lambda,a}^2 \to \rho_{\lambda,0}^2 & \mbox{a.e. in $\Omega$;} \\
		0 \leq \rho_{\lambda,a}^2 \leq |u_{\lambda,0}|^{q-1} & \mbox{a.e. in $\Omega$}.
	\end{array}\right.
\end{equation}
Moreover, since 
\begin{equation}
	\label{eq:embeddingrho2}
	\rho_{\lambda,a}^2(x) \in 
	\left\{
	\begin{array}{ll}
		\{0\}, & \text{if } u_{\lambda,a}(x) < a;\\
		\left[0,a^{q-1}\right], & \text{if } u_{\lambda,a}(x) = a;\\
		\{u_{\lambda,a}^{q-1}\}, & \text{if } u_{\lambda,a}(x) > a,
	\end{array}
	\right.
\end{equation}
then, up to a null measure set, if $u_{\lambda,0}(x) > 0$, then
\begin{equation}
	\label{eq:embeddingrho3}
	\rho_{\lambda,0}^2(x) = \lim_{a \to 0^+}\rho_{\lambda,a}^2(x) = u_{\lambda,0}^{q-1}.
\end{equation}

To establish the desired results, that is the Theorem \ref{Main1.1}, note that since $({\textbf z}_{\lambda,a}) \in L^\infty(\Omega,\mathbb{R}^N)$ and satisfies $\|{\textbf z}_{\lambda,a}\|_\infty \leq 1$, then there exists ${\textbf z}_{\lambda,0} \in L^\infty(\Omega,\mathbb{R}^N)$, such that
$$
{\textbf z}_{\lambda,a} \overset{*}{\rightharpoonup} {\textbf z}_{\lambda,0}, \quad \mbox{in $L^\infty(\Omega,\mathbb{R}^N)$.}
$$
For any test function $\varphi \in C^\infty_c(\Omega)$, this weak-$*$ convergence implies
$$
\int_\Omega {\textbf z}_{\lambda,a} \cdot \nabla \varphi dx \to \int_\Omega {\textbf z}_{\lambda,0} \cdot \nabla \varphi dx, \quad \mbox{as} \ a \to 0^+.
$$
Thus, we deduce that
\begin{equation}
	\label{eq:divdistribution}
	\mbox{div}\, {\textbf z}_{\lambda,a} \to \mbox{div}\, {\textbf z}_{\lambda,0}, \quad \mbox{in $\mathcal{D}'(\Omega)$}.
\end{equation}

Combining \eqref{eq:divdistribution} with \eqref{eq:embeddingrho1}, we conclude that
\begin{equation}
	\label{}
	-\mbox{div}\, {\textbf z}_{\lambda,0} = \lambda \rho_{\lambda,0}^1 + |u_{\lambda,0}|^{q-1}\, \quad \mbox{in $\mathcal{D}'(\Omega)$}.
\end{equation}

Next, it remains to verify the following key identities:
$$
(\mathbf{z}_{\lambda,0}, Du_{\lambda,0}) = |Du_{\lambda,0}|, \quad \text{in } \mathcal{M}(\Omega)
$$
and
$$
-u_{\lambda,0} \left[ \mathbf{z}_{\lambda,0},\nu \right] = |u_{\lambda,0}|, \quad \mathcal{H}^{N-1}\text{-a.e. on } \partial \Omega.
$$
The above results can be shown using arguments similar to those in \cite[see Lemmas 4.2 and 4.3]{PSJ}. From everything that has been commented so far, Theorem \ref{Main1.1} is proven.

\subsection{A critical perturbation of a discontinuous nonlinearity}\label{Ap2}	

In this subsection, we investigate the existence of multiple solutions for the following class of problems:  
$$
\left\{
\begin{array}{rclcl}
	-\Delta_1 u &=& \lambda H(|u|-a) |u|^{q-2}u + |u|^{1^*-2}u, &\text{in}& \Omega;\\
	u &=& 0, & \text{on}  & \partial\Omega, \\
\end{array}
\right.\eqno{(Q_{\lambda,a})}
$$
where $\Omega \subset \mathbb{R}^N$ is a bounded domain with smooth boundary, $N \ge 2$, $a > 0$, $\lambda > 0$, $q \in (1, 1^*)$, and $H$ denotes the Heaviside function defined in \eqref{Heaviside}. 

Given the discontinuity of the nonlinearity at the points $\pm a$, we begin, as before, by revisiting the notion of a solution to $(Q_{\lambda,a})$ and briefly examining the nature of the discontinuous term. 

Define  
$$
F_a(s) \coloneqq \int_{0}^{s} H(|t|-a) |t|^{q-2}t \, dt.
$$
For a fixed $a > 0$, using the same approach as in the previous subsection, it can be shown that $F_a$ is locally Lipschitz on $\mathbb{R}$, and its subdifferential is given by  
\begin{equation} \label{GradFa}
	\partial F_a(t) = \left\{ \begin{array}{rlc}
		\{|t|^{q-2}t\},& \mbox{if} \; |t| >a;& \\
		\{0\},& \mbox{if} \; |t| <a;&\\
		\left[0,a^{q-1}\right],& \mbox{if} \; t=a;& \\
		\left[-a^{q-1},0\right],& \mbox{if} \; t=-a.&
	\end{array}
	\right.
\end{equation}
This satisfies the growth condition  
\begin{equation}\label{growthgradf}
	|s| \le |t|^{q-1}, \; \; \mbox{for any} \ s \in \partial F_a(t) \subset \mathbb{R}.
\end{equation}
Additionally, $F_a$ is explicitly given by  
	\begin{equation}\label{growthFs>a}
	F_a(s) = \left\{ \begin{array}{rlc}
		\dfrac{|s|^q}{q} - \dfrac{a^q}{q},& \mbox{if} \; |s| \ge a;& \\
		0,& \mbox{if} \; |s| <a,&
	\end{array} \right.
\end{equation}
and satisfies the bound  
\begin{equation}\label{growthF}
	|F_a(s)| \le \dfrac{1}{q}|s|^q, \; \; \mbox{for any} \ s \in \mathbb{R}.
\end{equation}

These properties provide a foundation for understanding the nonlinear term in $(Q_{\lambda,a})$ and will play a crucial role in our analysis of solution multiplicity. Now, we can define the sense of solution of $(Q_{\lambda,a})$, as follows:
\begin{definition}\label{defsol}
	We say that $u_{\lambda,a} \in BV(\Omega)$ is a solution of $(Q_{\lambda,a})$, if there exist ${\bf z} \in L^\infty(\Omega;\mathbb{R}^N)$, with $\|{\bf z}\|_\infty \le 1$, and $\rho \in L^{p}(\Omega)$, $p=q/(q-1)$, such that
	$$
	\left \{
	\begin{array}{c}\displaystyle
		-\int_\Omega u_{\lambda,a}\ {\rm div} \, {\bf z} \ dx = \int_\Omega |Du_{\lambda,a}| + \int_{\partial\Omega} |u_{\lambda,a}| d \mathcal{H}^{N-1};\\
		-{\rm div} \, {\bf z} = \lambda\rho + |u_{\lambda,a}|^{1^*-2}u_{\lambda,a}\in L^N(\Omega), \quad \mbox{a.e. in} \ \Omega,
	\end{array}
	\right.
	$$
	with $\rho \in \partial F_a(u_{\lambda,a})$, a.e. in $\Omega$.
\end{definition}

In order to apply Theorem \ref{SP1} and establish the existence of solutions to $(Q_{\lambda,a})$ we define, for any $\lambda, a > 0$, the functional $J_{\lambda,a} : L^{1^*}(\Omega) \to \mathbb{R}$ as  
$$
J_{\lambda,a}(u) = \int_\Omega |Du| + \int_{\partial\Omega} |u| \, d\mathcal{H}^{N-1} - \lambda \int_\Omega F_a(u) \, dx - \frac{1}{1^*} \int_\Omega |u|^{1^*} \, dx.
$$  

To simplify the notation, we introduce the following functionals $\Psi : L^{1^*}(\Omega) \to \mathbb{R}$ and $Q : L^{1^*}(\Omega) \to \mathbb{R}$:
$$\Psi(u)\coloneqq	\left\{ \begin{array}{clc} \displaystyle
	\int_\Omega |Du| + \int_{\partial\Omega} |u|d\mathcal{H}^{N-1},& \mbox{if} \; u \in BV(\Omega);& \\
	+\infty,& \mbox{if} \; u \in L^{1^*}(\Omega) \setminus BV(\Omega),&
\end{array}
\right.$$
which is convex and lower semicontinuous, and  
$$
Q(u) \coloneqq -\frac{1}{1^*} \int_\Omega |u|^{1^*} \, dx,
$$  
which belongs to the $C^1$ class.

Additionally, for any $a > 0$, we define the functionals $\mathcal{I}_a : L^q(\Omega) \to \mathbb{R}$ and $I_a : L^{1^*}(\Omega) \to \mathbb{R}$ as  
$$
\mathcal{I}_a(u) = \int_\Omega F_a(u) \, dx \quad \text{and} \quad I_a \equiv \mathcal{I}_a\big|_{L^{1^*}(\Omega)}.
$$  
By Lemma \ref{incgrad1}, we have the following result:  
\begin{lemma} \label{incgrad}
	For any $a > 0$, the functionals $\mathcal{I}_a$ and $I_a$ are locally Lipschitz in $L^q(\Omega)$ and $L^{1^*}(\Omega)$, respectively. Furthermore:  
	\begin{enumerate}[$i)$]
		\item $\partial I_a(u) \subseteq \partial \mathcal{I}_a(u)$ for any $u \in L^{1^*}(\Omega)$;  
		\item $\partial \mathcal{I}_a(u) \subseteq \partial F_a(u)$ for any $u \in L^q(\Omega)$.
	\end{enumerate}
\end{lemma}

\begin{remark}
	The statement in item $ii)$ of Lemma \ref{incgrad} is a notational convenience. Specifically, this notation means that given $u \in L^q(\Omega)$ and $\xi \in \partial \mathcal{I}_a(u) \subset (L^q(\Omega))'$, there exists $\eta \in (L^q(\Omega))'$ such that  
	$$
	\langle \xi, v \rangle = \int_\Omega \eta v \, dx, \quad \text{for all } v \in (L^q(\Omega))',
	$$  
	and $\eta \in \partial F_a(u)$ almost everywhere in $\Omega$.
\end{remark}

Thus, we can rewrite the functional $J_{\lambda,a}$ as  
$$
J_{\lambda,a}(u) = \Psi(u) + \Phi_{\lambda,a}(u),
$$  
where $\Phi_{\lambda,a}(u) = Q(u) - \lambda I_a(u)$. This formulation implies that $J_{\lambda,a}$ satisfies property $(H_1)$.

Now we will prove that $J_{\lambda,a}$ works like the Euler-Lagrange functional associated to $(Q_{\lambda,a})$, i.e., any critical point of $J_{\lambda,a}$ is a solution of $(Q_{\lambda,a})$.

\begin{proposition}
	Assume that $u \in BV(\Omega)$ is a nontrivial critical point of $J_{\lambda,a}$. Then, $u$ is a solution of $(Q_{\lambda,a})$ in the sense of Definition \ref{defsol}.
\end{proposition}
\begin{proof}
	Let $u \in BV(\Omega)$ be a critical point of $J_{\lambda,a}$. Then, Definition \ref{cp} item $i)$ implies the existence of $w \in \partial\Psi(u)\subseteq L^N(\Omega)$ and $\xi\in\partial\Phi_{\lambda,a}(u)$ such that
	\begin{equation}\label{a}
		-\xi = w, \; \; \mbox{in} \; L^N(\Omega).
	\end{equation}
	Furthermore, by Lemma \ref{realiza}, there exists ${\bf z} \in L^\infty(\Omega,\mathbb{R}^N)$, with $\|{\bf z}\|_\infty \le 1$, satisfying
	\begin{equation}\label{b}
		\left \{
		\begin{array}{c}
			-\mbox{div}\, {\bf z} = w,\\
			\displaystyle \int_\Omega u w dx = \int_\Omega |Du| + \int_{\partial\Omega} |u| d \mathcal{H}^{N-1}.
		\end{array}
		\right.
	\end{equation}
	Lastly, by definition of $I_a$ and Lemma \ref{incgrad}, there exists $\rho \in L^{p}(\Omega)$ such that
	$$-\xi = \lambda\rho + |u|^{1^*-2}u,$$
	with $\rho \in \partial F_a(u)$, almost everywhere in $\Omega$. This last equality combined with \eqref{a} and \eqref{b} leads to
	$$-\mbox{div} \, {\bf z} = \lambda\rho + |u|^{1^*-2}u,\,\,\,\text{in}\,\,\,L^N(\Omega)$$
	and this complete the proof.
\end{proof}

Our goal now is to prove that $J_{\lambda,a}$ satisfies the $(\rm PS)$ condition at a controlled levels. To achieve this, we will introduce a series of lemmas that will help in this proof. To begin, let us define two auxiliary functions. For any $k>0$, we define $T_k,R_k : \mathbb{R} \to \mathbb{R}$ as follows:
$$T_k(s) = \min\big\{\max\{s,-k\},k\big\} \quad \mbox{and} \quad R_k(s) = s - T_k(s).$$
It is well known, by simple computation, that for any $u \in L^{1^*}(\Omega)$
$$T_k(u) \to u, \; \; \mbox{in} \  L^{1^*}(\Omega) \;\; \mbox{as} \ k \to \infty, \eqno{(T_1)}$$
and so, by definition of $R_k$
$$R_k(u) \to 0, \; \; \mbox{in} \  L^{1^*}(\Omega) \;\; \mbox{as} \ k \to \infty. \eqno{(R_1)}$$
Moreover, if $(u_n) \subset L^{1^*}(\Omega)$ is a sequence satisfying
$$u_n(x) \to u(x), \; \; \mbox{a.e. in} \ \Omega,$$
then, for each fixed $k>0$, by Lebesgue's Theorem	
$$T_k(u_n) \to T_k(u), \; \; \mbox{in} \  L^{1^*}(\Omega) \;\; \mbox{as} \ n \to \infty. \eqno{(T_2)}$$
Furthermore,
$$\Psi(u) = \Psi\big(T_k(u)\big) + \Psi\big(R_k(u)\big). \eqno{(R_2)}$$
Using this notation Degiovanni and Magrone \cite[Proposition 3.2]{DM} established the following lemma:
\begin{lemma}\label{DiMa}
	Let $(u_n)$ be a sequence in $BV(\Omega)$ and $(w_n)$ be a sequence in $L^N(\Omega)$ such that, for any $n \in \mathbb{N}$, $w_n \in \partial\Psi(u)$ and
	$$u_n \rightharpoonup u \quad \mbox{in} \ L^{1^*}(\Omega),$$
	$$w_n \rightharpoonup w \quad \mbox{in} \ L^{N}(\Omega),$$
	as $n \to \infty$. Then, $u \in BV(\Omega)$ and $w \in \partial\Psi(u)$.
\end{lemma}

\begin{lemma}\label{auxlemma}
	Let $(u_n)$ be a $(\rm PS)$ sequence for $J_{\lambda,a}$. Then, $(u_n)$ is bounded in $BV(\Omega)$ and
	$$u_n \rightharpoonup u \; \mbox{in} \; L^{1^*}(\Omega).$$
	Furthermore,
	$$\lim_{n \to \infty} \Big(\Psi(u_n) - \|u_n\|_{1^*}^{1^*}\Big) = \Psi(u) - \|u\|_{1^*}^{1^*}, \leqno{(a)}$$
	and, for any fixed $k>0$,
	$$\limsup_{n \to \infty} \Big(\Psi(R_k(u_n)) - \|R_k(u_n)\|_{1^*}^{1^*}\Big) \le \Psi(R_k(u)) - \|R_k(u)\|_{1^*}^{1^*}.\leqno{(b)}$$
\end{lemma}
\begin{proof}
	Let us start by proving that any $(\rm PS)$ sequence is bounded in $BV(\Omega)$. Let $(u_n)$ be a $(\rm PS)_d$ sequence of $J_{\lambda,a}$. Thus, $J_{\lambda,a}(u_n) \to d,$	and similar to \eqref{***} we get $\phi_n \in L^N(\Omega)$ with $\phi_n = o_n(1) \; \; \mbox{in} \ L^N(\Omega),$
	and
	$$\int_\Omega |D u_n| + \int_{\partial\Omega} |u_n| d \mathcal{H}^{N-1} = \lambda \int_\Omega \rho_n u_n dx + \int_\Omega |u_n|^{1^*}dx + \int_\Omega \phi_n u_n dx.$$
	In this case, defining
	$$A(u_n) \coloneqq \int_\Omega |D u_n| + \int_{\partial\Omega} |u_n| d \mathcal{H}^{N-1} - \left( \lambda \int_\Omega \rho_n u_n dx + \int_\Omega |u_n|^{1^*}dx + \int_\Omega \phi_n u_n dx\right),$$
	it follows that $A(u_n) = 0$ for any $n \in \mathbb{N}$. Thus,
	\begin{eqnarray}\label{Ilad}
		d + o_n(1) &=& J_{\lambda,a}(u_n) - \dfrac{1}{q} A(u_n) \nonumber \\
		&=& \left(1-\dfrac{1}{q}\right)\Psi(u_n) + \dfrac{\lambda}{q} \int_\Omega\big[\rho_nu_n- qF_a(u_n)\big] dx + \left(\dfrac{1}{q} - \dfrac{1}{1^*}\right)\|u_n\|_{1^*}^{1^*} + \dfrac{1}{q}\int_\Omega \phi_n u_ndx. \nonumber \\
	\end{eqnarray}
	Similarly to \eqref{lim5}, by using \eqref{GradFa}--\eqref{growthFs>a} we get
	\begin{equation}\label{rFge}
		\dfrac{\lambda}{q}\int_\Omega \big[\rho_nu_n- qF_a(u_n)\big] dx \ge -\dfrac{\lambda}{q} a^q |\Omega| \eqqcolon C_0.
	\end{equation}
	
	
	
	
	\noindent	By combining \eqref{Ilad} and \eqref{rFge} we obtain
	$$d+o_n(1) \ge \left(1-\dfrac{1}{q}\right)\Psi(u_n) -C_0 + \left(\dfrac{1}{q} - \dfrac{1}{1^*}\right)\|u_n\|_{1^*}^{1^*} - \dfrac{1}{q} \|\phi_n\|_N \|u_n\|_{1^*}.$$
	We can choose $n$ sufficiently large such that $\|\phi_n\|_N \leq (1 - q/1^*)$, and then
	$$d+o_n(1) \ge \left(1-\dfrac{1}{q}\right)\Psi(u_n) -C_0 + \left(\dfrac{1}{q} - \dfrac{1}{1^*}\right)\Big(\|u_n\|_{1^*}^{1^*} - \|u_n\|_{1^*}\Big),$$
	for $n$ large enough. Since the function $\zeta : [0,+\infty) \to \mathbb{R}$ defined by $\zeta(t) = t^{1^*} - t$ is bounded from below, i.e., there exists $K>0$ such that $\zeta(t) \ge -K$, for any $t \in [0,+\infty)$,	we conclude that
	$$d+o_n(1) \ge \left(1-\dfrac{1}{q}\right)\Psi(u_n) -C_0 - \left(\dfrac{1}{q} - \dfrac{1}{1^*}\right)K,$$
	which proves that $(u_n)$ is bounded in $BV(\Omega)$, as desired.

	We now proceed to establish item $(a)$ in the second part of the lemma. Since $(u_n)$ is a $(\rm PS)$ sequence, by Lemma \ref{caracPS} there exists $(\phi_n) \subset L^N(\Omega)$ with $\phi_n = o_n(1)$ and $\phi_n \in \partial \Psi(u_n) + \partial\Phi_{\lambda,a}(u_n)$.	So, by definition of $\Phi_{\lambda,a}$, we find $\rho_n \in \partial I_a(u_n) \subset L^N(\Omega)$ such that
		$$\lambda\rho_n + |u_n|^{1^*-2}u_n + \phi_n \in \partial\Psi(u_n),$$
		i.e., there exists $w_n \in \partial\Psi(u_n)$ such that
		\begin{equation}\label{*}
			w_n = \lambda\rho_n + |u_n|^{1^*-2}u_n + \phi_n.
		\end{equation}
		Moreover, by Lemma \ref{realiza}
		\begin{equation}\label{**}
			\int_\Omega w_n u_n dx = \int_\Omega |D u_n| + \int_{\partial\Omega} |u_n| d \mathcal{H}^{N-1}.
		\end{equation}
		By combining \eqref{*} and \eqref{**}, we get
		\begin{equation}\label{***}
			\int_\Omega |D u_n| + \int_{\partial\Omega} |u_n| d \mathcal{H}^{N-1} = \lambda \int_\Omega \rho_n u_n dx + \int_\Omega |u_n|^{1^*}dx + \int_\Omega \phi_n u_n dx.
	\end{equation}
	
	Since the sequence $(u_n)$ is bounded in $L^{1^*}(\Omega)$, Hölder's inequality ensures that $(|u_n|^{1^*-2}u_n)$ is also bounded in $L^N(\Omega)$ (see, for instance, \cite[Lemma 3.4]{AOP}). Furthermore, we claim that the sequence $(\rho_n) \subset \partial I_a(u_n)$ is bounded in $L^N(\Omega)$. 
	
	Indeed, by Lemma \ref{incgrad}, we have $\partial I_a(u_n) \subset \partial F_a(u_n)$, and from \eqref{growthgradf}, it follows that $|\rho_n| \le |u_n|^{q-1}$ almost everywhere in $\Omega$. Thus, applying Hölder's inequality, we obtain  
	$$
	\|\rho_n\|_N^N \le \left(\int_\Omega |u_n|^{N/(N-1)} dx\right)^{(N-1)(q-1)} |\Omega|^{1-(N-1)(q-1)},
	$$  
	which confirms the boundedness of $(\rho_n)$ in $L^N(\Omega)$, as required. 
	
	Given this boundedness and the reflexivity of $L^N(\Omega)$, there exist $u \in L^{1^*}(\Omega)$ and $\rho \in L^N(\Omega)$ such that  
	\begin{equation}\label{Rho}
	|u_n|^{1^*-2}u_n \rightharpoonup |u|^{1^*-2}u \quad \text{and} \quad \rho_n \rightharpoonup \rho,
	\end{equation}  
	both in $L^N(\Omega)$. Defining $w = \lambda \rho + |u|^{1^*-2}u$, these weak convergences imply $w_n \rightharpoonup w$ in $L^N(\Omega)$. 
	
	By hypothesis, $u_n \rightharpoonup u$ in $L^{1^*}(\Omega)$. Combining this with the $w_n$ convergence and Lemma \ref{DiMa}, we deduce that $u \in BV(\Omega)$ and  
	$$
	\lambda \rho + |u|^{1^*-2}u \in \partial\Psi(u).
	$$
	Thus, by Lemma \ref{realiza}, we have  
	\begin{equation}\label{eqnorm}
		\int_\Omega |Du| + \int_{\partial\Omega} |u| d \mathcal{H}^{N-1} = \lambda \int_\Omega \rho u \ dx + \int_\Omega |u|^{1^*}dx.
	\end{equation}
	
	On the other hand, by the compact embedding, $u_n \to u$ in $L^q(\Omega)$. Furthermore, by Lemma \ref{incgrad}, $(\rho_n) \subset \partial \mathcal{I}_a(u_n) \subset L^p(\Omega)$. Thus, applying \cite[Proposition 4]{SPS}, we infer that  
	$$
	\rho_n \rightharpoonup \widetilde{\rho} \quad \text{in } L^p(\Omega),
	$$  
	where $\widetilde{\rho} \in \partial \mathcal{I}_a(u)$. We claim that $\rho = \widetilde{\rho}$ a.e. in $\Omega.$	To prove this, note that the weak convergence of $\rho_n$ and \eqref{Rho} imply that  
	$$
	\int_\Omega \rho \varphi \ dx = \int_\Omega \widetilde{\rho} \varphi \ dx, \quad \forall \varphi \in C_0^\infty(\Omega),
	$$  
	which establishes the equality $\rho = \widetilde{\rho}$ almost everywhere. Consequently, we conclude that  
	\begin{equation}\label{convrn}
		\lim_{n \to \infty} \int_\Omega \rho_nu_n \ dx = \int_\Omega \rho u \ dx.
	\end{equation}
	
	Finally, combining \eqref{***}, \eqref{eqnorm}, and \eqref{convrn}, we obtain  
	\begin{align*}
		\lim_{n \to \infty} \Big(\Psi(u_n) - \|u_n\|_{1^*}^{1^*}\Big) &= \lim_{n \to \infty} \left(\lambda \int_\Omega \rho_n u_n \, dx + \int_\Omega \phi_n u_n \, dx\right) \\
		&= \lambda \int_\Omega \rho u \, dx \\
		&= \Big(\Psi(u) - \|u\|_{1^*}^{1^*}\Big),
	\end{align*}  
	which completes the proof of item $(a)$. The proof of item $(b)$ follows similarly to \cite[Lemma 5.1]{DM}.
\end{proof}

Now we are able to prove that $J_{\lambda,a}$ satisfies the $(\rm PS)$ condition. To do this, define
$$S = \inf_{\parbox{1.8cm}{\scriptsize\centering $u \in BV(\Omega)$  $u \neq 0$}} \dfrac{\displaystyle \int_\Omega |Du|}{\|u\|_{1^*}}.$$
\begin{lemma}\label{PScSN}
	For each $\lambda>0$, the functional $J_{\lambda,a}$ satisfies the $(\rm PS)_c$ condition for any
	$$c < \dfrac{1}{N} S^N.$$
\end{lemma}	
\begin{proof}
	Let $(u_n)$ be a $(\rm PS)_c$ sequence for $J_{\lambda,a}$, with $c \in (0,\frac{1}{N}S^N)$. This means
		\begin{equation}\label{convc}
			J_{\lambda,a}(u_n) \to c < \dfrac{1}{N} S^N, \; \; \mbox{as} \ n \to \infty,
		\end{equation}
		and
		$$\int_\Omega |D u_n| + \int_{\partial\Omega} |u_n| d \mathcal{H}^{N-1} = \lambda \int_\Omega \rho_n u_n dx + \int_\Omega |u_n|^{1^*}dx + \int_\Omega \phi_n u_n dx,$$
		with $\rho_n \in \partial F_a(u_n)$ and $\phi_n = o_n(1)$, where we used Lemma \ref{caracPS}. Expanding $J_{\lambda,a}(u_n)$, we find
	\begin{eqnarray}\label{bound1}
		J_{\lambda,a}(u_n) &=& \int_\Omega |D u_n| + \int_{\partial\Omega} |u_n| d \mathcal{H}^{N-1} - \lambda \int_\Omega F_a(u_n) dx - \dfrac{1}{1^*} \int_\Omega |u_n|^{1^*}dx \nonumber \\
		&=& \left(1-\dfrac{1}{1^*}\right)\|u_n\|_{1^*}^{1^*} + \lambda \int_\Omega \big[\rho_n u_n - F_a(u_n)\big]dx + o_n(1).
	\end{eqnarray}
	By direct computations, using \eqref{GradFa} and \eqref{growthF}, we deduce 
	$$\lambda \int_\Omega \big[\rho_n u_n - F_a(u_n)\big]dx \ge 0,$$
which combined with \eqref{bound1} lead us to
	\begin{equation}\label{ineun}
		\left(1-\dfrac{1}{1^*}\right)\|u_n\|_{1^*}^{1^*} \le J_{\lambda,a}(u_n) + o_n(1).
	\end{equation}
Since $(u_n)$ is bounded in $BV(\Omega)$ (see Lemma \ref{auxlemma}) and $BV(\Omega) \hookrightarrow L^{1^*}(\Omega)$ continuously, we get
	$$\sup_{n \in \mathbb{N}} \|u_n\|_{1^*}^{1^*} < \infty.$$
Taking the limit as $n\to\infty$ in \eqref{ineun} and using \eqref{convc}, we conclude
	\begin{equation}\label{*4}
		\lim_{n \to \infty} \|u_n\|_{1^*}^{1^*} \le Nc < S^{N},
	\end{equation}
after passing to a subsequence if necessary. Since $L^{1^*}(\Omega)$ is reflexive, there exists $u \in L^{1^*}(\Omega)$ such that $u_n \rightharpoonup u$ in $L^{1^*}(\Omega)$. Furthermore, by compact embedding $BV(\Omega) \hookrightarrow L^1(\Omega)$, we further have $u_n(x) \to u(x)$ a.e. in $\Omega$. By combining $(R_1)$, $(R_2)$, $(T_1)$ and the lower semicontinuity of $\Psi$ we get
	\begin{eqnarray}
		\limsup_{k \to \infty} \Big(\Psi\big(R_k(u)\big) - \|R_k(u)\|_{1^*}^{1^*}\Big) &=& \limsup_{k \to \infty} \Big(\Psi(u) - \Psi\big(T_k(u)\big) - \|R_k(u)\|_{1^*}^{1^*}\Big) \nonumber \\
		&=& \limsup_{k \to \infty} \Big(\Psi(u) - \Psi\big(T_k(u)\big)\Big) \nonumber \\
		&\le& 0. \nonumber
	\end{eqnarray} 
Thus, given $\varepsilon>0$, there exists $k>0$ large enough, such that
	\begin{equation}\label{*5}
		\Psi\big(R_k(u)\big) - \|R_k(u)\|_{1^*}^{1^*} < \varepsilon \Big(S-(Nc)^{\frac{1}{N}}\Big).
	\end{equation}
In addition, for fixed $k>0$ satisfying the inequality \eqref{*5}, the definition of $R_k$ and \eqref{*4} imply that
	\begin{equation}\label{*6}
		\limsup_{n \to \infty} \|R_k(u)\|_{1^*}^{1^*-1} \le \limsup_{n \to \infty} \|u\|_{1^*}^{1^*-1} \le (Nc)^{\frac{1}{N}}.
	\end{equation}
	By definition of the constant $S$, we know that
	$$S \le \dfrac{\displaystyle \int_\Omega |DR_k(u_n)|}{\|R_k(u_n)\|_{1^*}},$$
	and then
	\begin{eqnarray}
		\Big(S-\|R_k(u_n)\|_{1^*}^{1^*-1}\Big)\|R_k(u_n)\|_{1^*} &\le& \int_\Omega |DR_k(u_n)| -\|R_k(u_n)\|_{1^*}^{1^*} \nonumber \\
		&\le& \int_\Omega |DR_k(u_n)|+ \int_{\partial\Omega} |R_k(u_n)|d\mathcal{H}^{N-1} -\|R_k(u_n)\|_{1^*}^{1^*} \nonumber \\
		&=& \Psi\big(R_k(u_n)\big) - \|R_k(u_n)\|_{1^*}^{1^*}, \nonumber
	\end{eqnarray}
	thus, combining this previous inequaliy with Lemma \ref{auxlemma}, \eqref{*5} and \eqref{*6} we get
	\begin{eqnarray}
		\limsup_{n \to \infty} \|R_k(u_n)\|_{1^*} &\le& \limsup_{n \to \infty}\dfrac{\Psi\big(R_k(u_n)\big) - \|R_k(u_n)\|_{1^*}^{1^*}}{S-\|R_k(u_n)\|_{1^*}^{1^*-1}} \nonumber \\
		&\le& \dfrac{\Psi\big(R_k(u)\big) - \|R_k(u)\|_{1^*}^{1^*}}{S-(Nc)^{\frac{1}{N}}} \nonumber \\
		&\le& \varepsilon, \nonumber
	\end{eqnarray}
for $k>0$ large enough. Finally, by $(T_2)$ we obtain
	$$\limsup_{n \to \infty} \|u_n-u\|_{1^*} \le \limsup_{n \to \infty} \|T_k(u_n) - T_k(u)\|_{1^*} + \limsup_{n \to \infty} \|R_k(u_n)\|_{1^*} + \|R_k(u)\|_{1^*} \le 2\varepsilon.$$
Since $\varepsilon>0$ is arbitrary, the last inequality ensures that
	$$u_n \to u, \; \; \mbox{in} \ L^{1^*}(\Omega),$$
as desired.
\end{proof}

Now, we will prove that the functional $J_{\lambda,a}$ satisfies the geometry of Theorem \ref{SP1}. More precisely, the following lemma holds.

\begin{lemma}\label{geometry}
	The functional $J_{\lambda,a}$ satisfies the following properties:
	\begin{enumerate}[$(a)$]
		\item There exist $\tau,r>0$ such that
			$$J_{\lambda,a}(u) \ge \tau, \; \; \mbox{for} \ \|u\|_{1^*}=r.$$
			
		\item For each $n \in \mathbb{N}$ it is possible to find a $n$-dimensional subspace $X_n$ of $L^{1^*}(\Omega)$, with
			$$J_{\lambda,a}(u) \to -\infty \; \; \mbox{as} \ \|u\|_{1^*}\to \infty,$$
		with $u \in X_n$.
	\end{enumerate}
\end{lemma}
\begin{proof}
	\textit{Verification of $a)$}: First, if $u \in L^{1^*}(\Omega) \setminus BV(\Omega)$, by definition of $\Psi$ we have $J_{\lambda,a}(u)=+\infty$, which proves the geometry. In this case, we just need to consider $u \in BV(\Omega)$. By the continuous embeddings $BV(\Omega) \hookrightarrow L^{1^*}(\Omega) \hookrightarrow L^1(\Omega)$ and \eqref{growthF} we get
	\begin{eqnarray}
		J_{\lambda,a}(u) &=& \|u\|_{BV(\Omega)} - \dfrac{1}{1^*} \|u\|_{1^*}^{1^*} - \lambda\int_\Omega F_a(u)dx \nonumber \\
		&\ge& \|u\|_{BV(\Omega)} - \dfrac{1}{1^*} \|u\|_{1^*}^{1^*} - \dfrac{\lambda}{q}\|u\|_q^q \nonumber \\
		&\ge& C_1\|u\|_{{1^*}} - \dfrac{1}{1^*} \|u\|_{1^*}^{1^*} - C_2\dfrac{\lambda}{q}\|u\|_{1^*}^q. \nonumber
	\end{eqnarray}
	Since $1<q<1^*$ the last inequality allows us to conclude that there exist $r>0$ small enough such that
	$$J_{\lambda,a}(u) \ge \tau \; \; \mbox{for} \ \|u\|_{1^*}=r,$$
	where $\alpha \coloneqq   C_1r - \dfrac{1}{1^*} r^{1^*} - C_2\dfrac{\lambda}{q}r^q,$ and this proves the item $(a)$.

	\noindent \textit{Verification of $(b)$}: consider $X_n$ a $n$-dimensional subset of $L^{1^*}(\Omega)$, such that $X_n \subset C_0^\infty(\Omega)$. Thus, for each $u \in X_n$ we get
	\begin{eqnarray}
		J_{\lambda,a}(u) &=& \|u\|_{BV(\Omega)} - \dfrac{1}{1^*} \|u\|_{1^*}^{1^*} - \lambda\int_\Omega F_a(u)dx \nonumber \\
		&\stackrel{\eqref{growthF}}{\le}& \|u\|_{BV(\Omega)} - \dfrac{1}{1^*} \|u\|_{1^*}^{1^*} + \dfrac{\lambda}{q}\|u\|_q^q. \nonumber
	\end{eqnarray}
	Since in $X_n$ all the norms are equivalent, there exist positive constants $C_n,D_n$, which depend only on $n$, such that
	$$J_{\lambda,a}(u) \le C_n \|u\|_{1^*} - \dfrac{1}{1^*} \|u\|_{1^*}^{1^*} + D_n\dfrac{\lambda}{q}\|u\|_{1^*}^q,$$
	and the last inequality lead us to
	$$J_{\lambda,a}(u) \to -\infty \; \; \mbox{as} \ \|u\|_{1^*}\to \infty,$$
	provide that $1<q<1^*$, which completes the proof of item $(b)$, and lemma follows.
\end{proof}

The final task, before presenting the main result of this subsection, is to prove that in suitable conditions the functional $J_{\lambda,a}$ is below level $\frac{1}{N}S^N$.

\begin{lemma}\label{JPS}
	For each $n\in\mathbb{N}$, there exist $R_n>0$, sufficiently large, and $\lambda_n>0$  such that, for each $\lambda\in [\lambda_n,\infty)$, there is $a_\lambda>0$ satisfying
		$$\sup_{u \in S_n} J_{\lambda,a}(u) < \dfrac{1}{N} S^N,\quad  \forall a \in (0,a_\lambda),$$
	where $S_n= \overline{B}_{R_n}\cap X_n$, with $X_n$ satisfying Lemma \ref{geometry}.
\end{lemma}
\begin{proof}
	Indeed, fixed $n\in\mathbb{N}$, since $X_n$ is finite dimensional, we find a constant $D_n>0$, which depends only on $n$, such that
	$$J_{\lambda,a}(u) \le D_n \|u\|_{1^*} - \lambda\int_{\Omega} F_a(u)dx,$$
	for any $u \in X_n$. Now, for simplicity of notation, for fixed $u \in X_n$, denote
	$$\Omega_a \coloneqq \{x \in \Omega \; \; : \;\; |u|<a\} \; \; \mbox{and} \; \; \Omega^a \coloneqq \{x \in \Omega \; \; : \;\; |u|\ge a\}.$$
	Thus, by \eqref{growthFs>a} and \eqref{growthF} we get
	\begin{eqnarray}
		J_{\lambda,a}(u) &\le& D_n\|u\|_{1^*,\Omega^a} + D_n\|u\|_{1^*,\Omega_a} - \lambda\int_{\Omega^a} F_a(u)dx  - \lambda\int_{\Omega_a} F_a(u)dx \nonumber \\
		&\le&  D_n\|u\|_{1^*,\Omega^a} + D_n\|u\|_{1^*,\Omega_a} - \lambda\int_{\Omega^a} \left(\dfrac{|u|^q}{q} - \dfrac{a^q}{q}\right)dx + \dfrac{\lambda}{q}\int_{\Omega_a} |u|^qdx \nonumber \\
		&\le& \left(D_n\|u\|_{1^*,\Omega^a} - \dfrac{\lambda}{q}\|u\|_{q,\Omega^a}^q\right) + \left(D_n\|u\|_{1^*,\Omega_a} +\lambda\dfrac{a^q}{q}|\Omega|+ \dfrac{\lambda}{q}\|u\|_{q,\Omega_a}^q\right) \nonumber \\
		&\le& D_n\|u\|_{1^*,\Omega^a} - \dfrac{\lambda}{q}\|u\|_{q,\Omega^a}^q + c_{\lambda,a}, \nonumber
	\end{eqnarray}
	where
		$$c_{\lambda,a}\coloneqq |\Omega| \left(D_n a + 2\dfrac{\lambda}{q} a^q \right)>0.$$
	Using again the fact that $X_n$ is finite dimensional, there exists $E_n>0$, depending only on $n$, such that
		$$J_{\lambda,a}(u) \le D_n \|u\|_{1^*,\Omega^a} - E_n\dfrac{\lambda}{q}\|u\|_{1^*,\Omega^a}^q + c_{\lambda,a},$$
	which implies that
		\begin{equation}\label{supJ}
			\sup_{u \in S_n} J_{\lambda,a}(u) \le \sup_{u \in S_n} \Big\{D_n \|u\|_{1^*,\Omega^a} - E_n\dfrac{\lambda}{q}\|u\|_{1^*,\Omega^a}^q + c_{\lambda,a}\Big\}.
		\end{equation}
	Defining the function $\zeta_{\lambda,a} : (0,+\infty) \to \mathbb{R}$ by
	$$\zeta_{\lambda,a}(t) = D_nt - E_n\dfrac{\lambda}{q}t^q+c_{\lambda,a},$$ 
	by a direct computation we can see that
		$$\max_{t \ge 0} \zeta_{\lambda,a}(t) = C_n \left(\dfrac{1}{\lambda}\right)^{\frac{1}{q-1}} + c_{\lambda,a},$$
	where $C_n>0$ depends only on $n$ and $q$. In view of the defintion of $c_{\lambda, a}$ above, it is possible to find $\lambda_n>0$ large enough such that, given a $\lambda\geq \lambda_n$, it holds
		$$\max_{t \ge 0} \zeta_{\lambda,a}(t) < \dfrac{1}{N}S^N, \quad a \in (0,a_\lambda),$$
	for a suitable $a_\lambda>0$. This, combined with \eqref{supJ}, ensures that
		$$\sup_{u \in S_n} J_{\lambda,a}(u) < \dfrac{1}{N}S^N.$$
\end{proof}

Finally, we are able to prove the mains results of this subsection. Let us start proving the existence of multiple solutions of the problem $(Q_{\lambda,a})$ as follows:

\begin{proof}[Proof of Theorem \ref{Main2}]
	By Lemmas \ref{PScSN}, \ref{geometry}, and \ref{JPS}, for any $n \in \mathbb{N}$, we can find $\lambda_n > 0$ sufficiently large and $a_\lambda > 0$ such that for a fixed $\lambda \in [\lambda_n, \infty)$, the functional $J_{\lambda,a}$ satisfies the hypotheses of Theorem \ref{SP1} with $Z = L^{1^*}(\Omega)$ and $V = X_n$. Hence, by Theorem \ref{SP1}, it follows that the problem $(Q_{\lambda,a})$ admits at least $n$ nontrivial solutions, as required.
	
	Finally, using similar arguments as in the proof of Theorem \ref{Main1}, we establish $(M_2)$. This completes the proof.
\end{proof}

To analyze the stability of solutions, we proceed as in Subsection \ref{Ap1}. Specifically, consider the problem
	$$
	\left \{
	\begin{array}{rclcl}
		-\Delta_1 u &=& \lambda|u|^{q-2}u + |u|^{1^*-2}u, &\mbox{in}& \Omega;\\
		u &=& 0, & \mbox{on}  & \partial\Omega.\\
	\end{array}
	\right.\eqno{(Q_{\lambda})}
	$$
It is possible to establish results analogous to Lemmas \ref{Lemma.monotonicity} and \ref{lemma:bounded} in the framework of the functional $J_{\lambda,a}$. Using these, we prove the following theorem:

\begin{theorem}
	Fixed any $n \in \mathbb{N}$ consider $\lambda \in [\lambda_n, \infty)$ and $u_{\lambda,a}$ a solution to $(Q_{\lambda,q})$ obtained through Theorem \ref{Main2}. Then, there exists $u_{\lambda} \in BV(\Omega)$, such that, as $a \to 0^+$,
		$$
		u_{\lambda,a} \to u_{\lambda}, \quad \mbox{in $L^r(\Omega)$ and a.e. in $\Omega$,}
		$$
	for $r \in [1,1^*)$. Moreover, $u_{\lambda}$ is a bounded variation solution of $(Q_{\lambda})$.
\end{theorem}

\end{document}